\documentclass[11pt]{amsart}
\usepackage{amsmath,amsfonts,amscd,amssymb,amsthm}
\usepackage{graphicx}
\usepackage[a4paper,vmargin={2cm,2cm},hmargin={2cm,2cm},bottom=20mm]{geometry}
\usepackage{todonotes}

\usepackage{subcaption}
\numberwithin{equation}{section}
\usepackage{xfrac}
\usepackage[colorlinks,bookmarks]{hyperref}
\usepackage[mathscr]{euscript}

\newtheorem{theorem}{Theorem}[section]
\newtheorem{definition}[theorem]{Definition}

\newtheorem{corollary}[theorem]{Corollary}
\newtheorem{lemma}[theorem]{Lemma}
\newtheorem{proposition}[theorem]{Proposition}

\newtheorem{remark}[theorem]{Remark}

\newcounter{as}[section]

\newcommand{\mc}[1]{{\mathcal #1}}

\newcommand{\bb}[1]{{\mathbb #1}}

\newcommand{\<}{\langle}
\renewcommand{\>}{\rangle}
\renewcommand{\>}{\rangle}

\usepackage{color}
\definecolor{Red}{cmyk}{0,1,1,0}

\definecolor{Blue}{cmyk}{1,1,0,0}

\DeclareMathOperator{\sgn}{sgn}

\DeclareMathOperator{\supp}{supp}

\DeclareMathOperator{\Ci}{Ci}
\DeclareMathOperator{\Cin}{Cin}
\DeclareMathOperator{\spann}{span}
\def\1{{\mathchoice {\rm 1\mskip-4mu l} {\rm 1\mskip-4mu l}
{\rm 1\mskip-4.5mu l} {\rm 1\mskip-5mu l}}}

\title[Fractional Edgeworth expansions]{Fractional Edgeworth expansions  for one-dimensional heavy-tailed random variables and applications}


\author[L. Chiarini]{\small Leandro Chiarini}
\address{Utrecht University, Budapestlaan 6, 3584 CD Utrecht, The Netherlands}
\email{l.chiarinimedeiros@uu.nl}
\author[M. Jara]{\small Milton Jara}
\address{IMPA, Estrada Dona Castorina 110, 22460-320, Rio de Janeiro, Brazil}
\email{mjara@impa.br}
\author[W. M. Ruszel]{\small Wioletta M. Ruszel}
\address{Utrecht University, Budapestlaan 6, 3584 CD Utrecht, The Netherlands}
\email{w.m.ruszel@uu.nl}
\date{\today}

\begin{document}

\maketitle

\begin{abstract}
In this article, we study a class of lattice random variables in the domain of attraction of an $\alpha$-stable random variable with index $\alpha \in (0,2)$ which satisfy a truncated fractional Edgeworth expansion.
Our results include studying the class of such fractional Edgeworth expansions under simple operations, providing concrete examples; sharp rates of convergence to an $\alpha$-stable distribution in a local central limit theorem; Green's function expansions; and finally fluctuations of a class of discrete stochastic PDE's driven by the heavy-tailed random walks belonging to the class of fractional Edgeworth expansions.
\end{abstract}

\section{Introduction and overview of the results}

Edgeworth expansions refer to asymptotic expansions of the cumulative distribution function (CDF) $F_X(\cdot)$ of properly centered and rescaled random variables $X$ (under some moment assumptions) which control the error, see e.g.
\cite{edge} and references therein.
Although the central limit theorem characterises the limiting distribution and the Berry-Ess\'een theorem quantifies uniform error bounds, they are not able to capture important quantities of the distribution such as skewness or kurtosis.
Those quantities can be seen, again under some higher moment assumptions, in Edgeworth expansions which are polynomial series of the CDF with coefficients related to the cumulants.
As a consequence, one can obtain local central limit theorem convergence rates and potential kernel expansions.

If the sequence of random variables is in the domain of attraction of an $\alpha$-stable distribution with $\alpha \in (0,2)$, then the variance and even the mean might not exist, hence Edgeworth expansions in the classical sense are not well defined.

	To overcome the lack of moments of order greater or equal to $\alpha$ for random variables in the domain of attraction of an $\alpha$-stable distribution, other quantities were studied.
For instance, Bergstr\"om introduced in the 1950s the concept of \textit{pseudo-moments} and \textit{difference moments}, see \cite{bergstrom1953distribution}.
Pseudo-moments are not useful in the context of integer-valued random variables since they are infinite, see e.g. Lemma 2.7 in \cite{christoph1992convergence}.

Let us explain the concept of difference-moments in the following.
Denote by $F_{\bar{X}}(\cdot)$, the CDF of a random variable in the domain of attraction of an $\alpha$-stable random variable $\bar{X}$, $k \ge \alpha$ be an integer and $F_X(\cdot)$ the CDF of a (continuous) random variable $X$.
Although the k-th moment is not finite, it is possible to define for some $r > \alpha$ the quantity
\begin{equation}\label{eq:diff-moment} 
  \chi_r=  r\int_{\bb R} |x|^{r-1} |H(x)|dx < \infty,
\end{equation}
where $H(\cdot)$ is defined as $H(x):=F_X(x)-F_{\bar{X}}(x)$.
In this case, we have that for every integer $k < r$, the quantity 
 $\eta_k:=  -k\int_{\bb R} x^{k-1} dH(x)$ is well defined, implying the following expansion for the characteristic function of $X$ in terms of the characteristic function of $\bar{X}$ 
\begin{equation}\label{eq:arbitrary-exp} 
  \phi_X(\theta)= \phi_{\bar{X}}(\theta)
  +
  \sum_{k = 0}^{\lfloor r \rfloor} \eta_k\frac{(i\theta)^k}{k!}
  +
  \mc O(|\theta|^r)
  \text{ as } \theta \to 0. 
\end{equation}
In case the quantity $\eqref{eq:diff-moment}$ is not well-defined, similar expansions could be obtained using \textit{integral difference moments}, see  
\cite{christoph1992convergence} and references therein (particularly \cite{christoph1984asymptotic} for the case of integer-valued random variables).

Let us emphasize two shortcomings from the methods considered above.
The first is that it only allows integer powers of $\theta$, and therefore the difference $(\phi_X-\phi_{\bar{X}})$ needs to be differentiable up to a high order for the sum in the expansion to be non-trivial.
The second shortcoming comes from the necessity for closed expressions for $p_X(\cdot)$ and/or $p_{\bar{X}}(\cdot)$ in order to apply the methods above.
This translates into two possible restrictions: i) restricting the of choice of stable distribution to one of the cases for which we have an analytic expressions for $p_{\bar{X}}(\cdot)$; or ii) restricting the choice of discrete random walks to a specific discretisation of $p_{\bar{X}}(\cdot)$  given by 
\begin{equation}\label{eq:dist-from-integral} 
  p_X(x):=\int_{x-1/2}^{x+1/2} p_{\bar{X}}(y)dy. 
\end{equation} 

\subsection*{Overview of the results}
In this article, we seek to generalise such classes and provide simple examples with explicit distributions.
In particular, we will study \textit{fractional Edgeworth expansions} of the characteristic function of an integer-valued $X$ (instead of the corresponding CDF), generalizing the difference moments approach to not necessarily positive integer powers in the expansion.
More precisely, assume that the common characteristic function of the collection of integer-valued random variables $(X_n)_{n\geq 0}$ satisfies the following expansion with respect to $\alpha \in (0,2)$ and regularity set $R_{\alpha} \subset (\alpha, 2+\alpha):$
\begin{align}\label{intro-phi}
	\phi_X(\theta) = 1
	- \mu_\alpha |\theta|^\alpha  
	+ i\mu^\prime_\alpha \sgn(\theta) |\theta|^\alpha  
	+ \sum_{\beta \in R_\alpha} \mu_\beta |\theta|^{\beta}
	+ \sum_{\beta \in R_\alpha} i\mu^\prime_\beta \sgn(\theta)|\theta|^{\beta} 
	+ \mc O\Big(|\theta|^{2+\alpha}\Big)
\end{align}
as $ |\theta|\longrightarrow 0$ with constants $\mu_\alpha >0, \mu_\alpha^\prime, \mu_{\beta}, \mu_\beta^\prime  \in \bb R$.
We will refer to the constants $\mu_\beta,\mu^\prime_\beta$ for $\beta \in R_\alpha \cup \{\alpha\}$ as the \textit{fictional moment of order $k$} of $X$ (or of $p_X(\cdot)$).
Notice that \eqref{intro-phi} implies that the distribution associated to $\phi_X$ is in the domain of attraction of an $\alpha$-stable distribution. 
The concept of the regularity set $R_{\alpha}$ is roughly inspired by the \textit{index set} $A$, which appears in the definition of \textit{regularity structures} in \cite{hairer2015}, where $A$ encodes possible ``homogeneities" on the several levels of (ir)regularity of the objects being studied.

The truncation at level $2+\alpha$ is somewhat arbitrary and related to the concrete examples which we will discuss. 
We say that $p_X(\cdot)$  is admissible if its characteristic function satisfies \eqref{intro-phi},  we will denote this by $p_X \in \mc A$. 

Note that the term fractional cumulant appeared in the physics paper \cite{hazut2015fractional} for the first time when considering symmetric random variables whose density is defined as the inverse Fourier transform of a infinite series of fractional powers of $|\theta|$.
Notice that their results are based on a precise infinite series for the characteristic function rather than an approximated expression.

Our results can be categorized into four types of contributions. We will sketch informally the main results for each contribution  below.

\subsubsection*{1. Class of admissible probability mass functions}

In Section \ref{sec:example} we will study properties of the class of admissible probability mass functions or distributions for short.
In particular, we will show in Lemma \ref{lem:closed-by-operations} that the class is closed under operations such as convolution and taking convex combinations, i.e. for $p_1, p_2 \in \mathcal{A}$ we have that $p_1*p_2\in \mathcal{A}$ and $\delta p_1+(1-\delta)p_2 \in \mathcal{A}$ for $\delta \in [0,1]$.
This will be used in Proposition \ref{prop:adm-are-enough} to prove that any stable random variable $\bar{X}$ with parameters $(\alpha, \beta, \gamma, \varrho)$ (see Definition~\ref{def:stable}) can be approximated by an admissible distribution, i.e.
there exists $p_X \in \mathcal{A}$ such that for the corresponding characteristic function we have that 
\[
\phi_X(\theta) = \phi_{\bar{X}}(\theta) + o(|\theta|^{\alpha})
\]
for $\alpha \in (0,2)\setminus \{1\}$ and $\theta$ small enough.
The main example is discussed in Proposition~\ref{prop:asymp-phi-alpha}.
We consider the transition probability of a heavy-tailed random walk given by the transition probability
\[
p_{\alpha}(x,y) = \frac{c_{\alpha}}{|x-y|^{1+\alpha}}, \, \, x\neq y, \alpha \in (0,2)
\]
and show that $p_{\alpha} \in \mathcal{A}$, $R_{\alpha} = \{2\}$ and determine the fictional moments of the expansion. 

\subsubsection*{2. Local central limit theorem}
Central limit theorems and local central limit theorems (LCLT) are fundamental results in probability theory.
There exists a vast literature providing different types of LCLT results (or local stable limit theorems) in the stable setting with explicit and implicit convergence rates, e.g.~\cite{basu1976local, basu1979non,  BergerNote, cara, Gnedenko, ibra, Mineka,rvaceva1961domains, stone1965local}.

To our knowledge, the  best explicit non-uniform convergence rate for 1d absolutely continuous $X$ was proven in \cite{dattatraya1994non}, where the author showed under some integrability conditions on the characteristic function, that for any $\alpha \in (0,2)$:
\begin{equation} \label{eq:previous-lclt}
	|x|^\alpha \left| p^n_{X}(x) - p^n_{\bar{X}} (x)\right| 
	\leq C n^{\gamma},
\end{equation}
where $\bar{X}$ is the  stable distribution with index $\alpha$ where $\gamma=1- \frac{2}{\alpha}$ if $\alpha \in [1,2)$ and $\gamma = 1- \frac{1}{\alpha}$ if $\alpha \in (0,1)$.
As for uniform  bounds in $x$, without further assumptions in the law of the step distributions, one can use classical results of convergence of random variables (such as in \cite{rvaceva1961domains, stone1965local}) which imply that 
\begin{equation}\label{eq:previous-unif}
	n^{\frac{1}{\alpha}}\left|
	p^n_{X}(x) - p^n_{\bar{X}} (x)\right| = o (1).
\end{equation}
Our main result concerning sharp LCLT convergence rates is Theorem  \ref{thm:lclt}:
\[
  \sup_{x\in \bb Z} \left |p^n_{X}(x) - p^n_{\bar{X}} (x)\right| 
  \le C n^{-\frac{\beta_1+1-\alpha}{\alpha}}
\]
for some positive constant $C$, where $\beta_1 = \beta_1(R_{\alpha})$ depends on the regularity set $R_{\alpha}$.
In fact, we obtain in Corollary \ref{corol:assymp-repairable} that $\beta_1=2\alpha$ in the case that $\mu_{\beta}=0$ for all $\beta \notin \{\alpha, 2\}$.
Depending on the sign of $\mu_2$ either the original or final distribution has to be modified by a distribution of the ``correct"  finite variance random variable to prevail the strong convergence rate.
The introduced error is of order $\mathcal{O}(n^{-\frac{1}{\alpha} + (1-\frac{2}{\alpha})})$ which will vanish as $n\rightarrow \infty$.

The modification idea is natural and has shown to be very fruitful for example in \cite{Frometa2018} where the authors used it to obtain better convergence rates of a truncated Green's function in $\bb Z^2$.

\subsubsection*{3. Green's function estimates}
Concerning discrete potential kernel or Green's function behaviour there has been some asymptotic estimates obtained in \cite{amir2017, Bass,BergerNote, berger2019, Williamson} and \cite{Widom} in the continuum.
In \cite{Williamson}, the author proves that for $\alpha \in (0,2)$ the discrete potential kernel 
(see Definition~\ref{def:potential-green}) is asymptotic to $\|x\|^{d-\alpha}L(|x|)$ where $L(\cdot)$ is a slowly varying function, whereas \cite{berger2019} obtains similar asymptotics for processes on $\mathbb{Z}^d$ with index $\alpha =(\alpha_1,\dots,\alpha_d)$ and  $\alpha \in (0,2]^d$, $d\geq 1$.
In Theorem \ref{thm:Green-adm} we prove for $\alpha \in [1,2)$ that there exist explicit constants $C_{\alpha},C_0, C_{\delta}$ such that for $|x|\rightarrow \infty$ and $\delta := \min( R_\alpha )$ we have that there are constants $C_0,C_\delta$ such that
\begin{itemize}
\item[(i)] If $\delta < 2\alpha-1$, then 
	\[
		a_X(0,x)= C_\alpha |x|^{\alpha-1} 
		+
		C_{\delta} |x|^{2\alpha-\delta-1}
		+ \mc O (|x|^{2\alpha-\delta-1}),
	\]
\item[(ii)] if $\delta > 2\alpha-1$, then 
	\[
		a_X(0,x)=   C_\alpha |x|^{\alpha-1} 
		+C_0+
		o(1),
	\]
\item[(iii)] if $\delta = 2\alpha-1$, then
	\[
		a_X(0,x)=  C_\alpha |x|^{\alpha-1}  +
		C_{\delta}\log |x| + \mc O (1)
	\]
	as $x \to \infty$. 
\end{itemize}
In particular, for $R_{\alpha} \in \{\varnothing, \{2\}\}$, we provide in Theorems \ref{thm:Green} and \ref{thm:green-1} more explicit asymptotic expansions.
We also  provide similar results for the Green's function (see Definition~\ref{def:potential-green}) in the regime $\alpha \in (2/3,1)$ in Theorem~\ref{thm:green-less1}. 

The proofs of the potential kernel bounds are original and  they exploit the asymptotics of the characteristic function together with H\"older continuity instead of using the LCLT as a starting point like in the classical case \cite{Limic10}.

We believe our estimates would be useful to derive rates of convergence of the hitting measure from discrete to the continuous case.
This convergence was first proven in \cite{kesten1961theorem}, but as far as the authors are aware, has no quantitative estimates.
Quantitative estimates could allow us to study limit shape theorems for growth models, such as the internal diffusion limited aggregation (iDLA) driven by long-range random walks.

\subsubsection*{4. Fluctuations of the scaling limit of fractional Gaussian fields}

Finally, as an application of the fractional Edgeworth expansion, we look turn to discrete stochastic partial differential equations.

Consider fields $\Xi^m$ which solve the equation
\[
\begin{cases}
  \mc	L^m \Xi^m(x) = \xi^m(x)  - \sum_{z \in \bb T_m }  \xi^m(z), & \text{ if } x \in \bb T_m \\
  \sum_{x \in \bb T_m} \Xi^m(x)=0
\end{cases} 
\]
where $\bb T_m = 2\pi((-m/2,m/2]\cap\bb Z)$ is the discrete torus, $\mc L^m$ is the semigroup of the random walk obtained by periodising the distribution $p_X(\cdot)$ around $\bb T_m$, and $(\xi^m(x))_x$ is a collection of i.i.d. random variables with finite variance.
We know that $\Xi^m$ (when renormalised) converges to a fractional Gaussian field $\Xi_{\alpha} $ of order $\alpha$ i.e the solution of the equation in the continuum where the noise is substituted by the white-noise on the torus, \cite{chiarini2021odometer}.
We will show in Theorem~\ref{thm:converge-of-fields-elip}, if the first non-trivial fractional cumulant of $p_X$ is of order sufficiently small, for an appropriate coupling between $\Xi^m$ and $\Xi_{\alpha}$, 
\[
C_1 m^{\beta - \alpha} (\Xi^m - C_2 \Xi_{\alpha}) \longrightarrow \Xi_{2\alpha -\beta}
\]
in probability as $m\rightarrow \infty$ in some appropriate Sobolev space $H^{-s}(\bb T)$, for some chosen $s$ depending on $\alpha, \beta$ and $\bb T$ the continuous torus.

Furthermore, we present similar results for the field satisfying the parabolic version of the equation above in Theorem \ref{thm:converge-of-fields-parab}.

The novelties of the paper include the construction and exploration of important properties of fractional Edgeworth expansions for stable random variables. 
We study the effect of having non-trivial fictional moments of higher orders and obtain sharp convergence rates by developing a repairing approach of the corresponding distributions. 
Besides that, we believe that the application given in Section~\ref{sec:second-order-conv} opens the possibility of similar results for general continuous objects constructed as the scaling limit of fractional Gaussian fields, such as Gaussian multiplicative chaos, parabolic Anderson model, and solutions to non-linear equations.

\subsection*{Structure of the article}

In Section \ref{sec:def}, we provide the setting and introduce necessary definitions.
In Section  \ref{sec:results} we state our main Theorems.
The subsequent Section \ref{sec:general} contains a discussion about the results and possible generalizations in different directions.
Section \ref{sec:example} deals with determining the expansion of the characteristic function for an explicit example of a long-range random walk and showing that it falls into the class $\mathcal{A}$ we consider in this article.
Section \ref{sec:lclt} contains all proofs regarding LCLT's and in Section \ref{sec:Green} we demonstrate estimates on the discrete Green functions/potential kernels.
Finally, Section~\ref{sec:second-order-conv} deals with fluctuations of the scaling limits of fractional Gaussian fields.
Some technical lemmas are postponed to the Appendix.

\subsection*{Acknowlegments} 
The authors would like to thank S. Fr\'ometa for conversations about early versions of this article.
We would also like to thank the anonymous referees, their suggestions helped to significantly improve this article.
L. Chiarini was financially supported by CAPES and the NWO grant OCENW.KLEIN.083.
M. Jara was funded by the ERC Horizon 2020 grant 715734, the CNPq grant 305075/2017-9 and the FAPERJ grant E-29/203.012/201.

\section{Definitions}\label{sec:def}

In this section, we will introduce all necessary notation and define the main objects.
We will denote by $\bb T = (-\pi,\pi]$ the one-dimensional torus.
Given $z \in \bb R$ and $r >0$, we write $B_r(z)$ to denote the interval $(z-r/2,z+r/2]$ around $z$ with length $r$.
For $f,g:\mathbb{R} \rightarrow \mathbb{R}$ we write \[ f(x) \lesssim g(x) \] if there exists a constant $C>0$, which does not depend on $x$, such that $f(x) \leq C g(x)$, analogously for $\gtrsim$.
The functions $\lfloor \cdot \rfloor$ and $\lceil \cdot \rceil$ denote the floor and ceiling functions, respectively.
We will write $\bb Z^+:=\{0\}\cup \bb N$.

Given finite sets of positive real numbers $A,B \subset \bb R^+$, we define its sum by 
\[
	A+B :=\left\{ a+b: a \in A, b \in B \right\},
\]
and 
\[
	\spann ( A ) := \left\{ \sum_{a \in A} l_a a : l_a \in \bb Z^+, a \in A \right\}.
\]

Let $C^{k,\gamma} (\bb R)$ denote the space of function in $\bb R$ with $k \ge 0$ derivatives s.t the $k$-th derivative is $\gamma$-H\"older continuous with $\gamma \in (0,1]$.
We will denote by $C^{k,\gamma} (\bb T)$ the subspace of $C^{k,\gamma} (\bb R)$ composed by $2\pi$-periodic functions.
The notation $f \in C^{k,\gamma-} (\bb T)$ will be used for $f \in C^{k,\gamma- \varepsilon}(\bb T)$ for $\varepsilon \in (0,\gamma)$ sufficiently small.
Similarly we will use the short notation $f(x) = \mc O (|x|^{\beta \pm })$ for $f(x) = \mc O(|x|^{\beta \pm \varepsilon})$ for all $\varepsilon >0$ sufficiently small, where $\mc O (\cdot)$ is the standard big-O notation.

We will call  $\mc F(f) $  the Fourier transform of $f$ given by
\[
	\mc F (f)(\theta) := \int_{\bb R} f(x) e^{i \theta \cdot x} dx
\]
for $\theta \in \bb R$ resp.~$\mc F_{\bb T}$ for $k\in \bb N$
\[
	\mc F_{\bb T} (f) (k) := \int_{\bb T} f(x) e^{i k \cdot x} dx.
\]

Let $(X_i)_{i\in \mathbb{N}}$ be a sequence of i.i.d.~integer-valued random variables defined on some common probability space $(\Omega, \mathcal{A}, \mathbb{P})$.
Denote by $p_X(\cdot)$ the probability distribution of $X$, with support in $\bb Z$. 
 We write shorthand $X$ instead of $X_i$ when we refer to one single random variable.
Call $S_n:=\sum_{i=0}^n X_i$ its sum and abbreviate by $p^n_X(\cdot)$ the corresponding probability distribution.
Denote by
\[
	\phi_X(\theta) := \bb E\Big[ e^{i \theta \cdot X} \Big], \, \, \, \theta \in \mathbb{R}
\]
its common characteristic function.   

\begin{definition}\label{def:stable}
For all $\alpha \in (0,2)\setminus \{1\}$, $\beta \in [-1,1]$, $\gamma \in (0,\infty)$, and $\varrho \in (-\infty,\infty)$, call $\bar{X}$ the stable random variable with parameters $(\alpha,\beta,\gamma,\varrho)$ if its the characteristic function is given by
\begin{align*}
	\phi_{\alpha,\beta,\gamma,\varrho}(\theta)=\exp \left(i \theta \varrho-|\gamma \theta|^{\alpha}\left(1-i \beta \sgn(\theta) \tan\left(\frac{\pi\alpha}{2}\right)\right)\right),
\end{align*}
for every $\theta \in \bb R$. 
\end{definition}
We will call $\bar{X}$ a \textit{symmetric stable random variable} with index $\alpha \in (0,2)$ for short, if
\begin{equation}\label{eq:def-symmetric-cont}
	\phi_{\bar{X}} (\theta) = e^{-\mu_{\alpha} |\theta|^{\alpha}}.
\end{equation}
Observe that in that case, $\gamma$ satisfies $\gamma=(\mu_{\alpha})^{1/\alpha}$.
Let $p_{\bar{X}}(\cdot)$ denote the density resp. $p^n_{\bar{X}}(\cdot)$ its $n$-th convolution.

For $\alpha = 1$, we will only consider the case $\beta=0$, given by the expression
\begin{align*}
	\phi_{1,0 ,\gamma,\varrho}(\theta)=\exp \left(i \theta \varrho-|\gamma \theta|\right).
\end{align*}
In the following let us define the class of random variables which we will consider in this article.

\begin{definition} \label{regularity}
Let $\alpha \in (0,2]$ and let $R_\alpha \subset (\alpha,2+\alpha)\cup \left(\bb N \cap (2+\alpha)\right)$ be a finite set.
Denote by $p_X(\cdot)$ the probability distribution of a random variable $X$ with support in $\bb Z$. We say that $X$ admits a fractional Edgeworth expansion with index $\alpha$ and regularity set $R_\alpha$,  if its corresponding characteristic function $\phi_X(\theta)$ satisfies the following expansion
\begin{align} \label{def:char-alpha}
	\phi_X(\theta) = 1
	- \mu_\alpha |\theta|^\alpha  
	+ i\mu^\prime_\alpha \sgn(\theta) |\theta|^\alpha  
	+ \sum_{\beta \in R_\alpha} \mu_\beta |\theta|^{\beta}
	+ \sum_{\beta \in R_\alpha} i\mu^\prime_\beta \sgn(\theta)|\theta|^{\beta} 
	+ \mc O\Big(|\theta|^{2+\alpha}\Big)
\end{align}
as $|\theta| \longrightarrow 0$, for constants $\mu_\alpha >0$ and $\mu^\prime_\alpha, \mu_\beta, \mu^\prime_\beta \in \bb R\setminus \{0\}$, for all $\beta \in R_\alpha$.
For $\alpha=1$,  we further require the law of $p_X(\cdot)$ to be symmetric.
The constants $\mu_\beta,\mu_\beta^\prime$ are referred to as the \textit{fictional moments of order $\beta$ of $p_X(\cdot)$}. Equivalently we will also simply say that $p_X(\cdot)$ is \textit{admissible} or $p_X \in \mc A$. 
\end{definition}

We will always assume that the set $R_\alpha$ is optimal, i.e, $|\mu_\beta|+|\mu^\prime_\beta|>0$ for all $\beta \in R_\alpha$. 
It is important to recall that the constants $\mu_{\alpha}, \mu^\prime_\alpha,\mu_\beta,\mu^\prime_\beta$, given in the definition above, depend on the law of $p_X(\cdot)$.

Using the expansion given in \eqref{def:char-alpha} and the Taylor polynomial of $\log (1+t)$ for $|t|<1$, setting $t:=\phi_X(\theta)-1$, we get that $\phi_X(\cdot)$ can be written as
\begin{align} \label{def:char-alpha-2}
		\phi_X(\theta) = e^{-\mu_\alpha |\theta|^\alpha 
		+
		r_X(\theta)
		+\mc O(|\theta|^{2+\alpha})},\quad\text{as } |\theta|\longrightarrow 0,
\end{align}
where
\[
	r_{X}(\theta)
	=
	\sum_{\beta \in  J_\alpha} \kappa_\beta|\theta|^{\beta}
	+
	\sum_{\beta \in  \{\alpha\} \cup J_\alpha} i \sgn(\theta)  \kappa_\beta^\prime |\theta|^{\beta},
\]
and the coefficients $\kappa_j$ are combinations of coefficients coming from the
expansion of the logarithm and  the powers $|\theta|^{\alpha}$
resp.~$|\theta|^{\beta}$.  
We will refer to $\kappa_\beta$ for all $\beta \in J_\alpha$ as the \textit{fractional cumulants of} $p_X(\cdot)$. 

Furthermore, let 
\begin{equation}\label{def:Jalpha}
	J_\alpha :=  \spann (R^+_{\alpha} ) \cap (\alpha,2+\alpha),
\end{equation}
where $R^+_{\alpha} := R_\alpha\cup \{\alpha\}$.
In a similar way we define $J_\alpha^+:= J_\alpha \cup \{2+\alpha\}$.

Remark that if $R_{\alpha}=\emptyset$ we have that $J_{\alpha} = \alpha \bb N \cap (\alpha, 2+\alpha)$ and therefore, in general we have $\beta_1 :=\min (J^+_{\alpha}) \leq 2\alpha$. 
Our \textit{regularity set} $R_{\alpha}$ is a finite collection of powers of $|\theta|$ in the expansion of the characteristic function, up to orders which are strictly smaller than $2+\alpha$.

Our main example of an admissible distribution with index $\alpha \in (0,2)$ and $R_{\alpha}=\{2\}$ is given by
\begin{equation} \label{def:palpha}
	p_{\alpha}(x)  :=
	\begin{cases}
		c_\alpha |x|^{-(1+\alpha)}, &\text{ if } x \neq 0,\\
		0,&\text{ if } x=0,
	\end{cases}
\end{equation}
where $c_{\alpha}$ is the normalising constant. 
We will discuss this example and many others in Section \ref{sec:example}.

An example of a distribution which is \textit{not admissible} is $p_\alpha(\cdot)$, defined in \eqref{def:palpha} with $\alpha=2$.
In fact, in this case the characteristic function has the expansion 
\[
	\phi_X(\theta) = 1- \mu_2|\theta|^2 \log (|\theta|) + \mc O (|\theta|^2).
\]
In \cite{nandori2011recurrence}, the author discusses some properties of this particular example including its recurrence and LCLT estimates.

In order to explore the idea of improving rates of convergence of a given random variable, we will concentrate on a particular subset of admissible distributions.
Then we will subdivide the class of admissible distributions in a subclass w.r.t.~regularity sets $R_{\alpha}\in \{ \varnothing, \{2\}\}$ and a subclass w.r.t.~general $R_{\alpha}$.
The first subclass will be further subdivided in three classes which will have different asymptotic behaviour as $n \rightarrow \infty$.
 
\begin{definition}
Let $X$ be a symmetric random variable such that $p_X \in  \mc A$ and  $R_\alpha \in \{\emptyset,\{2\}\}$. 
Then we say that $p_X(\cdot)$ belongs to one of the following three classes:
\begin{enumerate}
	\item[(i)] \textit{repaired} if  $R_\alpha =\emptyset$
	\item[(ii)] \textit{locally repairable}  if $R_\alpha=\{2\}$ and  $\mu_2 >0$
	\item[(iii)] \textit{asymptotically repairable} if $R_\alpha=\{2\}$ and  $\mu_2 <0$.
\end{enumerate}
\end{definition}

A \textit{locally repairable} probability distribution $p_X(\cdot)$ can be \textit{repaired} by convoluting it with a simple discrete random variable with variance $2|\mu_2|$ which plays the part of a \textit{repairer}.
Analogously, we can repair an \textit{asymptotically repairable} probability distribution $p_X(\cdot)$.
This repairing is not performed on $p_X(\cdot)$ itself.
Instead, we repair its asymptotic distribution $p_{\bar{X}}(\cdot)$ by convoluting $\bar{X}$ with a normal random variable with variance $2|\mu_2|$.
In both cases, the aim is to change either the original random variable $X$ or its stable limit $\bar{X}$ in order to cancel the contribution from $\mu_2$.

\begin{definition} \label{def:repairer}
Let $p_X(\cdot)$ be admissible with index $\alpha \in (0,2)$ with regularity set
$R_{\alpha}\in \{ \varnothing, \{2\}\}$ and let $\mu_2$ be the constant
defined in the expansion of $\phi_X(\cdot)$. 
\begin{itemize}
\item[(i)] If $p_X(\cdot)$ is locally repairable, we call the repairer  an
  independent random variable $Z$ with probability distribution given by
\begin{equation} \label{eq:def-repairer}
	p_Z(x) =
	\begin{cases}
		\frac{\mu_2}{M^2}, & \text{ if } |x|=M\\
		1-\frac{2\mu_2}{M^2}, & \text{ if } x=0\\
		0, & \text{	otherwise,}
	\end{cases}
\end{equation}
where $M = \lceil \sqrt{2 \mu_2}\rceil \in \bb N$.
\item[(ii)] If $p_X(\cdot)$ is asymptotically repairable,  we call an asymptotic repairer a random variable $\bar{Z}$ such that $\bar{Z} \sim \mathcal{N}(0, 2|\mu_2|)$.
We will assume that $\bar{Z}$ is independent of $\bar{X}$, whose characteristic function is given by \eqref{eq:def-symmetric-cont}. 
 \end{itemize}
\end{definition}
By construction,  the characteristic function of a repairer $Z$ has the expansion
\begin{align*}
	\phi_{Z}(\theta)
	 &=
	1- \mu_2 |\theta|^2 + \mc O (\theta^4), \qquad\text{ as } |\theta| \longrightarrow 0.
\end{align*}
It is easy to see that $p_{X+Z}(\cdot) = p_X*p_Z(\cdot)$ is in fact repaired.
The asymptotic repairer $\bar{Z}$ is such that the characteristic function of
$\bar{X}+\bar{Z}$ is equal to
\[
	\phi_{\bar{X}+ \bar{Z}}(\theta) = e^{-\mu_{\alpha} |\theta|^{\alpha} -\mu_2 |\theta|^2}.
\]

Note that in both cases we do not change the limiting distribution of $n^{-1/\alpha} S_n$.
Indeed,  this modification will introduce an error of order $\mathcal{O}(n^{1- \frac{3}{\alpha}})$ which vanishes as $n\rightarrow \infty$.

Let us remark that alternatively one could \textit{repair} by taking a convex combination as in \cite{Frometa2018}.
Different repairing methods might be more convenient depending on the context.

The idea of repairing random variables is reminiscent of the \textit{Lindeberg principle}, see \cite{linde} for the original article and \cite{lindestable} for a proof of the stable law by Lindeberg principle.
In the classical setting, its main idea is to explore the effect of swapping each discrete random variable by a Gaussian one with the correct mean and variance.
However, in this case, we are trying to match the discrete and the continuous random variable up to a higher order (fictional) moment.
This is done either by perturbing the law of the discrete or of the continuous random variable by a random variable which is strictly more regular (in the sense that it is integrable up to a higher order).
We believe that this can be extended to match not only one extra fictional moment (as in the case of the repairable random variables) but any finite Edgeworth expansion.
In Section~\ref{sec:general}, we discuss this idea further.

Finally, let us define the potential kernel for a random walk, whose transition probability  $p_X(\cdot) := p_X(\cdot, \cdot)$ is admissible with index $\alpha \in [1,2)$ and regularity set $R_{\alpha}$.

\begin{definition} \label{def:potential-green} 
Let $(X_i)_{i\in \mathbb{N}}$ be a sequence of, i.i.d.~random variables such that $p_X(\cdot)$ is admissible.
Call $S_n=\sum_{i=1}^n X_i$ and $p^n_X(\cdot, \cdot)$ its transition probability.
If $(S_n)_{n \in \bb N}$ is recurrent, we define the \textit{potential kernel} of $p_X(\cdot)$ as 
\[
	a_{X}(0,x) = \sum_{n=0}^{\infty}(p_X^n(0,x) - p_X^n(0,0)), \, \, \text{ for } x\in \bb Z.
\]
If the random walk $(S_n)_{n \in \bb N}$ is transient, we define the \textit{Green's function} of $p_X(\cdot)$ as 
\[
	g_{X}(0,x) = \sum_{n=0}^{\infty}p_X^n(0,x), \, \, \text{ for } x\in \bb Z.
\]
Henceforth, we will abbreviate $p_X(x):= p_X(0,x)$.
We will also write $a_X(x)=a_{X}(0,x)$ and $g_X(x)=g_{X}(0,x)$. 
\end{definition}

We will need a few more definitions to study the scaling limits of discrete PDEs.
For $m \ge 1$, let $\bb T_m:= 2\pi \bb  Z_m = 2\pi((-m/2,m/2]\cap \bb Z)$.
For $f_1,f_2 \in  L^2(\bb T)$, denote their inner product by
\begin{align*}
  \< f_1,f_2 \>:= 
  \< f_1,f_2 \>_{L^2(\bb T)} = 
  \frac{1}{2\pi}
  \int_{\bb T} f_1(x) \overline{f_2(x)} dx.
\end{align*}
We extend this notation to the case in which $f$ is a distribution and $g$ is a suitable test function as the action of $f$ on $g$. 
We then abuse the notation to also describe the inner product  of two functions in $\ell^2(\bb T_m)$ as 
\begin{align*}
  \< f_1,f_2 \> = 
  \< f_1,f_2 \>_{\ell^2(\bb T_m)} = 
  \frac{1}{2\pi m}
  \sum_{x \in \bb T_m} 
  f_1(x) \overline{f_2(x)}.
\end{align*}
This abuse of notation comes from the fact that given two functions $f_1, f_2\in C(\bb T)$, let $f_i^m$ be the restriction of $f_i$ to $\bb T_m$, then we have that $\< f_1^m, f_2^m \>_{\ell^2(\bb T_m)}\to \< f_1, f_2 \>_{L^2(\bb T)}$.
Moreover, both spaces share the same orthonormal basis given by Fourier functions.
That is, consider $\{{\bf e}_k\}_{k \in \bb Z}$ be given by ${\bf e}_k(x):= \exp\left(i k\cdot x\right)$, a orthonormal basis of $L^2(\bb T)$. 
Likewise, consider $\{{\bf e}^m_k\}_{k \in \bb Z_m}$ where ${\bf e}^m_k$ be given by the restriction of ${\bf e}_k$ to $\bb T_m$, this collection forms an orthonormal basis of $\ell^2(\bb T_m)$ .

For $s \ge 0$, consider the Hilbert space $H^{s}(\bb T)$ induced by the following inner product
\begin{equation*}
  \< f,g \>_{H^{s}}
:=
  \< f, g \>_{L^2} 
  +
  \sum_{k \in \bb Z}
  \mc F_{\bb T}(f)(k)
  \overline{\mc F_{\bb T}(g)(k)}
  |k|^{4s}.
\end{equation*}
The subset of functions which has $\mc F_{\bb T}(f)(0)= \int_{\bb T} f(x)dx = 0 $ is denoted by $H^s_0$. 
When studying $H^s_0$, we will use the norm induced by
\begin{equation*}
  \< f,g \>_{H^{s}_0}
:=
  \sum_{k \in \bb Z}
  \mc F_{\bb T}(f)(k)
  \overline{\mc F_{\bb T}(g)(k)}
  |k|^{4s}.
\end{equation*}
which is equivalent to the $H^s$-norm in $H^s_0$. 
We also consider the space $H^{-s}$ - the dual space of $H^{s}_0$ - with norm 
\begin{equation*}
  \|f\|_{H^{-s}}
:=
  \sum_{k \in \bb Z\setminus \{0\}}
  |\mc F_{\bb T}(f)(k)|^2
  |k|^{-4s}.
\end{equation*}

Finally, for any Hilbert space $H$ and $T>0$ , we denote by $L^2([0,T],H)$ the Hilbert space of functions $f:[0,T] \mapsto H$ with the norm
\begin{equation}\label{eq:L20T-norm}
  \|f\|_{L^2([0,T],H)}
:=
  \int_{0}^T 
  \|f(t)\|^2_H dt,
\end{equation}
where $\|.\|_H$ is the norm of $H$. 
\section{Results}\label{sec:results}
\subsection{Local central limit theorem}
In this section, we state our results regarding LCLT's for heavy-tailed i.i.d.~random variables with admissible probability distributions.

\begin{theorem} \label{thm:lclt}
Let $\alpha \in (0,2)$ and $(X_i)_{i\in \mathbb{N}}$ be a sequence of i.i.d.~random variables with admissible law $p_X(\cdot)$.
Let furthermore $p_{\bar{X}}(\cdot)$ denote the law of the $\alpha$-stable random variable satisfying Definition \ref{def:stable} to which domain of attraction belongs $X$.
Then we have that,
\begin{equation*}
	\sup_{x \in \bb Z}
	|{p}^n_X(x)- p^n_{\bar{X}}(x)|
\lesssim
	n^{- \frac{\beta_{1}+1-\alpha }{\alpha}},
\end{equation*} 
where $\beta_1=\min(J^+_{\alpha})$.
%
\end{theorem}

\begin{corollary}\label{corol:assymp-repairable}
Let $\alpha \in (0,2)$ and $(X_i)_{i\in \mathbb{N}}$ be a sequence of symmetric i.i.d.~random variables with asymptotically repairable law $p_X(\cdot)$.
Let furthermore $p_{\bar{X}}(\cdot)$ denote the law of the $\alpha$-stable
to which domain of attraction belongs $X$, then
\begin{equation*}
	\sup_{x \in \bb Z}
	|{p}^n_{X}(x)- p^n_{\bar{X} + \bar{Z}}(x)|
\lesssim
	n^{- (1+\frac{1}{\alpha})}.
\end{equation*}
\end{corollary}

The next theorem provides a more detailed expansion of the error obtained in the previous result.

\begin{theorem}\label{prop:lclt}
Let $p_X(\cdot)$ and $p_{\bar{X}}(\cdot)$ be distributions satisfying the assumptions of Theorem~\ref{thm:lclt}.
Then, there exists a collection of constants $\{C_{\beta},\beta \in J_\alpha\}$ s.t. for all $x\in \bb Z$,
      \begin{equation}
	      \left| p^n_X(x)- p^n_{\bar{X}}(x)
		  - \sum_{\beta \in J_\alpha}
	      C_\beta\frac{u_\beta\left( \frac{x}{n^{1/\alpha}} \right)}{n^{(1+\beta-\alpha)/\alpha}}
		  - \sum_{\beta \in J_\alpha}
	      C^\prime_\beta\frac{u^\prime_\beta\left( \frac{x}{n^{1/\alpha}} \right)}{n^{(1+\beta-\alpha)/\alpha}}
	       \right|
	      \lesssim
	      n^{-\frac{3}{\alpha}},
	      \label{eq:thm-lclt-beyond}
      \end{equation}
      \label{thm:lclt-beyond}
      where
      \begin{equation}\label{eq:inte}
	      u_\beta(x):=
	      \frac{1}{2\pi}
	      \int_{\bb R} 
		       |\theta|^{\beta} 
			   \phi_{\bar{X}}(\theta)
		      e^{-i x \theta }
	      d\theta
		  \text{ and }
	      u^\prime_\beta(x):=
	      \frac{1}{2\pi}
	      \int_{\bb R} 
		       \sgn(\theta) |\theta|^{\beta} 
			   \phi_{\bar{X}}(\theta)
		      e^{-i x \theta }
	      d\theta.
      \end{equation}
\end{theorem}
A careful analysis of the functions $u_\beta,u^\prime_\beta$ shows that 
\begin{equation}
	|u_\beta(x)|,|u^\prime_\beta(x)| \lesssim
	\frac{1}{|x|^{\alpha+\beta+1}}
	\label{eq:bound-uj}.
\end{equation}
Indeed, this bound is significantly weaker than its equivalent Theorem 2.3.7 in \cite{Limic10} in the case that $X$ has at least four finite moments.
There, the integrands in \eqref{eq:inte} are given by $g_\beta(\theta):=\theta^\beta e^{-c |\theta|^2}$, and therefore $g_\beta(\cdot)$ can be seen as Schwartz functions with rapidly decaying derivatives.

A simple triangular inequality leads us to the following corollary.

\begin{corollary}\label{coro}
	Under the conditions of Theorem \ref{prop:lclt}, calling
	$\beta_2 := \min ( {J}^+_{\alpha}\setminus \{\beta_1\} )$,  we have that 
\[
	\left |	p^n_X(x)- p^n_{\bar{X}}(x) \right |
 = o \left (\sum_{\beta \in J_\alpha}
		C_\beta\frac{u_\beta\left( \frac{x}{n^{1/\alpha}} \right)}{n^{(1+\beta-\alpha)/\alpha}} \right ).
\]
In particular, we have that
\begin{equation*}
\sup_{x\in \bb Z}	\left|
		p^n_X(x)- p^n_{\bar{X}}(x)
	\right|
	\lesssim
	n^{-\frac{(\beta_1+1-\alpha)}{\alpha}}
\end{equation*}
and
\begin{equation*}
	\left|
		p^n_X(x)- p^n_{\bar{X}}(x)
	\right|
	\lesssim
\left (n^{-\frac{(\beta_2+1-\alpha)}{\alpha}} \right )
	\vee
\left (	n^{2} |x|^{-(\alpha+\beta_1+1)} \right ).
\end{equation*}
\end{corollary}

Note that from Corollary \ref{coro} we can deduce that the rate of convergence given in Theorem~\ref{thm:lclt} is optimal.
Indeed, for any $\alpha \in (0,2)$, consider $\beta_1 \in (\alpha,2 \wedge 2\alpha)$, using  Lemma~\ref{lem:closed-by-operations} and Proposition~\ref{prop:asymp-phi-alpha-asymmetric} for $\alpha$ and $\beta_1$ we have that $p_X(\cdot):=p_{\alpha}*p_{\beta_1}(\cdot)$ is admissible with index $\alpha$ and regularity set $\{\beta_1, 2\}$.
Applying Corollary~\ref{coro}, we have that the error term in Theorem~\ref{prop:lclt} at $x=0$ is given by
\begin{equation*}
	u_{\beta_1}(0)n^{-\gamma}+ o \left(n^{-\gamma}\right)
\end{equation*}
where $\gamma =\frac{\beta_1 + 1-\alpha}{\alpha}$, and $u_{\beta_1}(0)>0$.

Notice that this reinforces this idea that repairing improves the convergence behaviour. 
If $p_X(\cdot)$ is repaired, then  $\beta_1 =  \min \{ 2\alpha, 2+\alpha \} = 2\alpha \geq 2$ which leads to $\gamma = \frac{\alpha+1}{\alpha}$.
For $\alpha \geq 1$ and $p_X(\cdot)$ is locally repairable we have that $\beta_1=\min \{2, 2\alpha, 2+\alpha \}=2$.
Without repairing, the best uniform bound we can get is 
\begin{equation*}
	\left|
		p^n_X(x)- p^n_{\bar{X}}(x)
	\right|
	\lesssim
	n^{1-3/\alpha},
\end{equation*}
which is much weaker than the bound in Theorem \ref{thm:lclt}, especially for $\alpha$ close to 2.
Theorem \ref{thm:lclt} states that repairing a probability distribution preserves the convergence rates.
Note that for $\alpha < 1$, we have that  $\beta_1 < 2$ so repairing will not provide better convergence bounds beyond the one in Corollary \ref{coro}.

In Section \ref{sec:general}, we discuss how one could potentially repair a distribution using heavy-tailed random variables instead of random variables with finite variance.

\subsection{Potential kernel estimates for long-range random walks}\label{sec:res_pot_kernel}

Theorem \ref{thm:Green} presents potential kernel estimates for long-range random walks with admissible law $p_X(\cdot)$.
It exemplifies that repairing distributions provides good potential kernel expansions.
This will be proven in Section \ref{sec:Green}.
Note that the results in this section hold for $\alpha \in (2/3,2)$.
For further considerations on $\alpha \le  2/3$, we refer to Section \ref{sec:general}.
Unfortunately, as our techniques require several cancellations, we are only able to proceed in the symmetrical case.

We will first treat the case $\alpha \in (1,2)$ and $\alpha=1$ for the subclass described by $R_{\alpha}\in \{ \varnothing, \{2\}\}$  separately.
To start, we give bounds for repaired distributions when $\alpha \in (1,2)$, where we have an expansion up to some vanishing error as $|x|\rightarrow \infty$.
After that we compute all terms of the expansion for locally and asymptotically repairable distributions up to order $\mc O (1)$.
Then, we present the general admissible case, in which we obtain the first and second terms of the expansion which will depend on $\delta:=\min(R_\alpha)$.
Finally, we look at the case $\alpha \in (2/3,1)$. 

\begin{theorem}\label{thm:Green}
Let $\alpha \in (1,2)$ and $(X_i)_{i\in \mathbb{N}}$ be a sequence of i.i.d.~random variables with common admissible distribution $p_X(\cdot)$ with index $\alpha$ and regularity set $R_{\alpha}\in \{ \varnothing, \{2\}\}$. 
\begin{itemize}
\item[(i)] Assume that $p_X(\cdot)$ is repaired, then there exist constants
  $C_{0}, C_{\alpha} \in \bb R$ such that
	\[
		a_X(x)=    C_\alpha |x|^{\alpha-1}  + C_0+
		\mc O(|x|^{\frac{\alpha-2}{3}+})
	\]
	as $x \to \infty$, where
	\[
		C_\alpha = \frac{1}{\pi \mu_\alpha} 
		\int_{0}^\infty \frac{\cos (\theta)-1}{\theta^\alpha}d\theta
	\]
	and 
	\[
		C_{0}
	= - \frac{\pi^{1-\alpha}}{2 \pi \mu_\alpha(\alpha-1)}+ 
		\frac{1}{\pi} \int_{0}^{\pi}
		\frac{\phi_X (\theta) -\left(  1-\mu_\alpha \theta^\alpha \right)}{ \mu_\alpha \theta^\alpha (1- \phi_X(\theta))}
		d\theta.
	\]
\item[(ii)] Assume that  $p_X(\cdot)$ is locally or asymptotically repairable. Let $m_\alpha:= \lceil \frac{\alpha-1}{2-\alpha}\rceil-1$, then there
	exist constants $C^\prime_0,C_1,\dots,C_{m_\alpha+1}$ such that 
	\[
		a_X(x)= C_\alpha |x|^{\alpha-1} 
		+
		\sum_{m=1}^{m_\alpha} C_m |x|^{(\alpha-1) - m (2-\alpha)}
		+
		C^\prime_0 \log{|x|}
		+
		\mc O(1)
	\]
	as $|x| \to \infty$, where for $1 \le m \le m_\alpha+1$ 
	\[
		C_m := \frac{\mu_2^m}{\pi \mu_\alpha^{m+1}}
		\int_{0}^\infty \theta^{m(2-\alpha)-\alpha} (\cos(\theta)-1 ) d\theta,
	\]
	and the sum is zero if $m_\alpha =0$. Moreover,
	\begin{equation*}
		C^\prime_0:= 
		\begin{cases}
			0,& \text{if } \frac{2}{2-\alpha} \not \in \bb N \\
			C_{m_\alpha+1},& \text{if } \frac{2}{2-\alpha} \in \bb N.
		\end{cases}
	\end{equation*}
\end{itemize}
\end{theorem}

Note that $m_\alpha \rightarrow \infty$ as $\alpha \rightarrow 2$, 
therefore, repairing (whenever possible) becomes more relevant for values of $\alpha$ close to $2$.
The following theorem treats the general admissible case.

\begin{theorem}\label{thm:Green-adm}
Let $\alpha \in (1,2)$ and $(X_i)_{i\in \mathbb{N}}$ be a sequence of i.i.d.~random variables with common admissible distribution $p_X(\cdot)$ with index $\alpha$ and regularity set $R_\alpha$. Let $\delta := \min( R_\alpha )$ and
	\[
		C_\alpha = \frac{1}{\pi \mu_\alpha} 
		\int_{0}^\infty \frac{ \cos (\theta) -1}{\theta^\alpha}d\theta.
	\]
\begin{itemize}
\item[(i)] If $\delta < 2\alpha-1$, then there 
	exists a  constant $C_{\delta}$ such that 
	\[
		a_X(x)= C_\alpha |x|^{\alpha-1} 
		+
		C_{\delta} |x|^{2\alpha-\delta-1}
		+ \mc O (|x|^{2\alpha-\delta-1})
	\]
	as $|x| \to \infty$, where 
	\[
		C_{\delta}= \frac{\mu_{\delta}}{\pi \mu_\alpha}
		\int_{0}^\infty \theta^{\delta-2\alpha} (\cos(\theta) -1) d\theta.
	\]
\item[(ii)] If $\delta > 2\alpha-1$, then there 
	exists a  constant $C_0$ such that
	\[
		a_X(x)=   C_\alpha |x|^{\alpha-1} 
		+C_0+
		o(1)
	\]
	as $x \to \infty$, where
	\[
		C_{0}
	= -\frac{\pi^{1-\alpha}}{2 \pi \mu_\alpha(\alpha-1)}+ 
		\frac{1}{\pi} \int_{0}^{\pi}
		\frac{\phi_X (\theta) -\left(  1-\mu_\alpha \theta^\alpha \right)}{ \mu_\alpha \theta^\alpha (1- \phi_X(\theta))}
		d\theta.
	\]
\item[(iii)] If $\delta = 2\alpha-1$, then there 
	exists a constant $C_{\delta}$ such that
	\[
		a_X(x)=  C_\alpha |x|^{\alpha-1}  +
		C_{\delta}\log |x| + \mc O (1)
	\]
	as $x \to \infty$, where 
	\[
		C_{\delta}:= \frac{\mu_{\delta}}{\pi \mu_\alpha} \int_{0}^\pi \frac{\cos(\theta)-1}{\theta} d\theta 
	\]
	\end{itemize}
\end{theorem}

Finally, we include the result for the potential kernel for $\alpha=1$,
when $R_{\alpha}\in \{ \varnothing, \{2\}\}$.

\begin{theorem}	\label{thm:green-1}
	Let $\alpha =1$ and $(X_i)_{i\in \mathbb{N}}$ be a sequence of i.i.d.~random variables with common admissible law $p_X(\cdot)$ and $R_{\alpha}\in \{ \varnothing, \{2\}\}$.
	Then
	\[
		a_{X}(x) = - \frac{1}{\pi \mu_1} \log(|x|)  
		+ C_0 +
		o \left( 1 \right),
	\]
	where 
	\[
		C_0:= 
		\frac{\gamma + \log (\pi)}{\pi\mu_1}
		+
		\frac{1}{\pi}
		\int_{0}^{\pi} \left(  \frac{1}{1-\phi_X(\theta) } - \frac{1}{\mu_1 |\theta|} \right)d\theta
	\]
	and  $\gamma$ is the Euler-Mascheroni constant.
	Additionally, if $p_X(\cdot)$ is repaired, we have that the term $o (1)$ is in 
	fact of order $\mc O \left( |x|^{-\frac{1}{3}+} \right)$.
\end{theorem}

\begin{theorem}	\label{thm:green-less1}
	Let $\alpha \in (2/3,1)$ and $(X_i)_{i\in \mathbb{N}}$ be a sequence of i.i.d.~random variables with common repaired law $p_X(\cdot)$ and $R_{\alpha} = \varnothing$.
	Then
	\[
		g_{X}(x) = C_\alpha |x|^{\alpha-1} +
		\mc	O (|x|^{-\alpha \frac{2-\alpha}{2+\alpha}-}).
	\]
\end{theorem}

\subsection{Fluctuations of discrete stochastic PDEs}
\label{sec:fluctuations-of-scaling-limits-of-discrete-stochastic-pdes}
As an application of our results we want to study second-order fluctuations of fractional Gaussian fields in this section.

We will work on the torus $\bb T=(-\pi,\pi]$. 
To keep the definitions consistent to the literature on fractional Gaussian fields, in this section, we will restrict ourselves to random walks with symmetric transition probabilities.

Let $\Xi^s=(\Xi^s(x))_{x\in \bb T}$ denote the so-called fractional Gaussian field of order $s \in \bb R$ (or $s$-FGF for short), on the torus, i.e, the solution of the elliptic equation 
\begin{equation}\label{eq:def-continuous-field}
  \begin{cases}
  -(-\Delta)^{s/2}_{\bb T}\Xi_{s}(x) = \xi(x) - \< \xi, 1 \> & \text{ if } x \in \bb T \\
  \phantom{-} \int_{\bb T} \Xi_{s}(x) dx =0,
  \end{cases} 
\end{equation}
where $\xi=(\xi(x))_{x\in \bb T}$ is the white-noise defined as $\<\xi, f\> \sim \mathcal{N}(0,\|f\|^2_{L^2(\bb T)})$ for $f\in C^{\infty}(\bb T)$, and $-(-\Delta)_{\bb T}^{s/2}$ is defined as the fractional Laplacian on the torus in terms of the Fourier transform, that is,
\begin{equation*}
  \mc F_{\bb T}( -(-\Delta)^{s/2}_{\bb T} f)(\theta)
  =
  \mc F_{\bb T}(f)(\theta)|\theta|^{2s},
\end{equation*}
for a given function $f \in C^\infty(\bb T)$.
We can then extend this operator to any distribution on $\bb T$ via duality. 
For more information on fractional Gaussian fields in a more general context, we suggest \cite{Lodhia2016}. 

We will compare this field with its discrete counterpart.
To do so, let $p_X(\cdot)$ be the transition probability of a symmetric admissible random walk on $\bb Z$ with index $\alpha$ , therefore the constants $\mu_\beta^\prime$ vanish.
In particular, let us write the characteristic	function of $X$ as 
\begin{equation}\label{eq:char-func-for-fields}
  \phi_X(\theta)=1-\mu_\alpha |\theta|^\alpha+ \mu_\beta |\theta|^\beta +\mc O(|\theta|^\gamma),
\end{equation}
for $\alpha, \beta, \gamma$ satisfying
\begin{equation}\label{eq:range-field}
  \alpha \in (0,2], \alpha < \beta < \gamma \le 2+\alpha,\text{  and }\beta <\alpha +1. 
\end{equation}

We embed such distributions on the discrete torus $\bb T_m:= 2\pi \bb  Z_m = 2\pi((-m/2,m/2]\cap \bb Z)$ by defining
\begin{equation}\label{eq:periodisation}
  p^{(m)}(x,y) = p_{X,m}(0,y-x):= \sum_{z \in \bb Z} p_X\left(\frac{y-x + 2\pi m z}{2\pi}\right),
\end{equation}
for $x,y \in \bb T_m$.
From this, we construct $\mc L^{m}$ the generator of the random walk, that is,
\begin{equation}\label{eq:def-generator}
  \mc L^m f(x): = m^\alpha\sum_{y \in \bb T_m} p^{(m)}(x,y) (f(y)-f(x)),
\end{equation}
where the factor $m^\alpha$ is just the scaling required to have $\mc L^{m}$ converge to $-\mu_\alpha (-\Delta)_{\bb T}^{\alpha/2}$ as $m\rightarrow
\infty$. 

Now, for a given $m \ge 1$,  we define the discrete fractional Gaussian field $h^m=(h^m(x))_{x\in \bb T_m}$ as the solution to 
\begin{equation}\label{eq:def-discrete-field-1} 
  \begin{cases}
  \mc L^{m}h^m(x) = \xi^m(x) - \frac{1}{m}\sum_{z \in \bb T_m} \xi^m (z), &  \text{ if } x\in \bb T_m \\
  \sum_{x \in \bb T_m}h^m(x)=0,
  \end{cases} 
\end{equation}
where $\{\xi^m(x)\}_{x \in \bb T_m}$ is a collection of i.i.d. $\mc N(0,1)$-random variables. 
We can then associate a distribution $\Xi^m$ to $h^m$ on $\bb T_m$ by
\begin{equation}\label{eq:def-discrete-field-2} 
  \Xi^m:= m^{-1/2}\sum_{x \in \bb T_m} \delta_x h^m(x),
\end{equation}
where $\delta_x$ denotes the Dirac's delta function.

\begin{theorem}\label{thm:converge-of-fields-elip}
Let $h^m$ be the solution of \eqref{eq:def-discrete-field-1}, $\Xi^m$ defined in \eqref{eq:def-discrete-field-2} and the parameters $\alpha, \beta, \gamma$ satisfying \eqref{eq:range-field}. 
Then, we can couple $\Xi^m$ and $\Xi_\alpha$ in such a way that 
\begin{equation}\label{eq:thm-conv-fields}
 m^{\beta-\alpha}  \frac{(\mu_\alpha)^2}{\mu_\beta}
 \left(\Xi^m -  \frac{1}{\mu_\alpha}\Xi_\alpha\right)
 \longrightarrow 
 \Xi_{2\alpha-\beta}
 \text{ in probability in }
 H^{-s}(\bb T),
\end{equation}
for all $s > s_0:= \max\{(3+2\alpha)/4, (\alpha-\beta)/2 +  (\beta-\alpha)\vee (\gamma-\beta)\}$ as $m\rightarrow \infty$.
\end{theorem}
Similarly, we will describe the parabolic counterpart of the result above.
Consider $(\zeta^m(t,x))_{t\geq 0, x\in \bb T_m}$ to be the solution of the discrete stochastic heat equation (SHE)
\begin{equation}\label{eq:parabolic-discrete}
  \begin{cases}
    d\zeta^m(t,x) 
	=
	\mc L^{m} \zeta^m(t,x)  dt + d\xi^m(t,x),
& t>0, x\in \bb T_m \\
\zeta^m(0,\cdot):=\zeta^m_0(x):= Z_0 \mid_{\bb T_m}, & t=0
  \end{cases} 
\end{equation}
where $(\xi^m(\cdot,x))_{x \in \bb T_m}$ is a family of independent Brownian motions on $\bb T_m$, $Z_0(x)$ is a smooth function in the continuous torus, and $Z_0 \mid_{\bb T_m}$ is the restriction to the discretized torus $\bb T_m$.

Again, we compare it with its continuous counterpart given by
\begin{equation}\label{eq:parabolic-conti}
  \begin{cases}
    dZ_\alpha(t,x) 
	=
	-\mu_\alpha(-\Delta)^{\alpha/2}Z_\alpha(t,x)dt + d\xi(t,x),
& t>0, x\in \bb T \\
Z_\alpha(0,\cdot):=Z_0(x), & t=0
  \end{cases} 
\end{equation}
which satisfies 
\begin{align*}
 \hat{Z}_\alpha(t,k) 
:=
\hat{Z}_0(k)e^{-\mu_\alpha|k|^\alpha t}
+
 \int_{0}^t e^{-\mu_\alpha|k|^\alpha (t-s)}  
 \hat{\xi}(ds,k),
\end{align*}
where this Fourier transform is only taken in the space coordinates. Consider the field
\begin{equation}\label{eq:parabolic-discrete-field}
  Z^m(t,\cdot) 
  =
  m^{-1/2}
  \sum_{x \in \bb T_m}
  \zeta^m(t,x)\delta_x.
\end{equation}
Then, the following theorem holds.

\begin{theorem}\label{thm:converge-of-fields-parab}
Let $\zeta^m$ solve \eqref{eq:parabolic-discrete} and $Z^m$ be defined in Equation \eqref{eq:parabolic-discrete-field}. 
Then, we can couple $Z^m$ and $Z_\alpha$ in such a way that for every $T>0$, 
\begin{equation}\label{eq:conv-fields-parab}
 m^{\beta-\alpha} 
 \left(
	Z^m -  Z_{\alpha}
 \right)
 \longrightarrow 
 Z_{Err}
 \text{ in probability in }
 L^2([0,T],H^{-s})
\end{equation}
for all $s > s_0:= \max \{2\beta-\alpha,\gamma-\alpha\}$ as $m\rightarrow \infty$, and where the field $Z_{Err}$ can be characterised by its Fourier coefficients
\begin{equation}\label{eq:Fourier-of-parab-error}
  \hat{Z}_{Err}(t,k):=
  -\mu_\beta|k|^\beta\hat{Z}_0(k)e^{-\mu_\alpha |k|^\alpha t}t
 - \mu_\beta\int_{0}^t e^{-\mu_\alpha |k|^\alpha (t-s)} |k|^\beta (t-s)\hat{\xi}(ds,k), 
\end{equation}
and
\begin{equation*}
  \hat{Z}_{Err}(t,0)\equiv 0,
\end{equation*}
where for each $k \in \bb Z \setminus \{0\}$, $\xi(\cdot,k)$ is an i.i.d. Brownian motion.
\end{theorem}

Notice that this is \textbf{not} the solution of a linear stochastic differential equation.
In fact, $Z_{Err}(t,\cdot)\to 0$ as $t\to 0$ for any initial condition and but the initial condition.
However, the initial condition has has a non-trivial influence over $Z_{Err}(t,\cdot)$, showing that the characterisation of the fluctuations is not simple.

A natural follow up question is the fluctuations around non-linear equations. We will discuss ideas that could be used to analyse such cases in the next section.

\begin{remark}
To evaluate fluctuations, we used the norm $H^{-s}$ on the space variable, which is not sensitive to the $0$-th Fourier coefficient of the test function.
However, notice that due to our choice of discretisation, $\widehat{Z}^n(t,0)\equiv\hat{Z}(t,0)$ for every $n$, and therefore $Z_{Err}(t,0)\equiv 0$.
A similar consideration applies to the elliptic case.
\end{remark} 


\section{Discussion and generalisations of our results} \label{sec:general}

In this section we quickly discuss possible generalisations and limitations of our results and techniques.
\subsection*{Admissible distributions, regularity sets $R_{\alpha}$ and error terms}

As mentioned in the introduction, the order $2+\alpha$ is chosen because it appears naturally in the examples we study, see Section~\ref{sec:example}.
In order to just obtain sharp convergence rates of the LCLT, expansions up to an error term of order $\mc O (|\theta|^{2\alpha})$ are enough.
Analogously, all of our other results benefit from the further order terms.
Regarding the potential kernel estimates in Section \ref{sec:res_pot_kernel}, choosing an error of order $\mc O(|\theta|^{2+\alpha})$ improves the expansion compared to choosing $\mc O (|\theta|^{2\alpha})$.
Furthermore, if we assumed Edgeworth expansions to orders beyond $2+\alpha$, we would also be able to generalise our results.

Remark that random variables with support in $\bb Z$ and finite fourth moment have an admissible distribution with index $\alpha=2$ and $R_{\alpha} \subset \{1,2,3\}$.
Both LCLT and potential kernel estimates for such random variables are well understood, see \cite{Limic10}.
For this reason, we concentrate on the case $\alpha \in (0,2)$.

The class of admissible probability distributions is closed under natural operations such as convex sum and convolution, see Lemma~\ref{lem:closed-by-operations}.

\subsection*{Higher dimensions and asymmetric Green's functions}

We can use a similar approach to the one used here - together with the multidimensional version of Euler-Maclaurin (see \cite{karshon2003euler}) and the Faà di Bruno's formula - to show that the random variable $X$ with distribution $p_X(x)=c_{\alpha,d}|x|^{-(d+\alpha)} \1_{x \neq 0}$ admits a fractional Edgeworth expansion for any dimension $d$.

Although we can obtain such expansions, we are not able to use the same analysis to control the signs of the constants $\mu_{1+\alpha}$ and $\mu_{2}$ that appear in the expansion, and therefore not able to use the techniques given here to improve rates of convergence and other results.
However, both the LCLT and fluctuations results can be generalised to this case with almost no additional changes.

Unfortunately, our results do not generalise to Green's function estimates for $d\ge 2$ and $\alpha \in (0,2)$ without further assumptions on the degree of continuity of the remainder of the function $\tilde{h}_X(\cdot)$. We would need to guarantee that the remainder would decay faster than $\|x\|^{\alpha-d}$, which is the first order term.

The same limitation applies to $\alpha <2/3$ and $d=1$, the degree of continuity of $\phi_X(\cdot)$ becomes too low to guarantee that its Fourier transform will decay faster than $|x|^{\alpha-1}$.

One could try to expand ideas from the proof of Theorem 1.4 in \cite{Frometa2018} to tackle the $d\geq 2$ and/or $\alpha < 1$ case.
There the authors demonstrate a detailed expansion for the Green's function in $d=2$, $\alpha \in (0,2)$ for a truncated long-range random walk.

Regarding adding asymmetry in the random walk, other methods would be required as the criteria we use for evaluating smoothness of the integrands loses continuity at $\theta=0$. 


\subsection*{Further repairers}

In this article we only studied repairers for probability distributions $p_X(\cdot)$ which are $\alpha$-admissible with a regularity set $R_\alpha = \left\{ 2 \right\}$.
However, suppose that $p_X(\cdot)$ is an admissible distribution, let $\delta:=\min ( R_\alpha )$ and $\mu_{\delta}>0$ so we are in the locally repairable case.
We could define a repairer $Z$ as an admissible distribution $p_Z(\cdot)$ with index $\delta$ whose leading coefficient  in the expansion of the characteristic function of $X$ is equal to the negative value of the coefficient $\mu_{\delta}$ multiplying $|\theta|^{\delta}$ in the expansion of the characteristic function of $Z$.

Then, $\min (R'_{\alpha} ) >\delta$, where $R'_{\alpha}$ is the regularity set of $X+Z$.
Hence repairing would allow to obtain more precise estimates on its potential kernel beyond the constant order of the error.
A similar idea could be used to improve the rates of convergence in the LCLT for distributions such that $\min (R_{\alpha}) < 2\alpha$, by performing multiple repairs to cancel each of the terms in $r_X(\theta)$.

However, if the constant $\mu_\delta$ is negative, one can also elaborate a similar notion for asymptotic repairer.
Again, in the same spirit of adapting the Lidenberg principle but matching the fictional moments/fractional cumulants.

\subsection*{Non-lattice walks/ Random variables in $\bb R$}

We believe that a combination of the ideas of the present paper and \cite{stone1965local} would be enough to prove our results in the context of non-lattice walks and absolutely continuous random variables, possibly depending on a further integrability assumption over the characteristic function.
However, we cannot say the same about potential kernel estimates.
Here we are relying on the fact that smoothness implies decay of the Fourier coefficients on the torus.
This relation fails in the infinite plane as the functions in the integrand are not smooth at zero.

\subsection*{Discrete fractional fields and fluctuations of discrete stochastic heat equations}
In \cite{Cipriani2016, chiarini2021odometer}, the authors proved that, after suitable rescaling, the discrete fractional Gaussian fields, respectively driven by the simple random walk and by random walk with a power-law decay, converge in distribution to their suitable continuous counterparts.
In fact, they proved it in any dimension and only assuming that the family $\{\xi^{m}(x)\}_{x \in \bb T_m}$ is i.i.d. (not necessarily normal) with finite variance.
Theorem \ref{thm:converge-of-fields-elip} shows that the fluctuations around such fields happen on the scale $m^{\alpha-\beta}$.
We can also characterise this fluctuation as another fractional Gaussian field of parameter $2\alpha-\beta$.
Notice that the exponent $2\alpha-\beta$ matches precisely the second order term in the heat-kernel expansion given in Theorem~\ref{thm:Green-adm}.

This leads us to think that similar results should hold in the whole space (rather than the torus). The reason for which we chose the torus is that the field is well-defined for any $\alpha \in (0,2)$ for any dimension $d \ge 1$ in the discrete torus; whereas the equivalent discrete fractional Gaussian on the full space only exists if $d >\alpha$, a regime for which our Green's function bounds are not strong enough to identify the correct order of fluctuations.

The assumption in Theorem \ref{thm:converge-of-fields-elip} of $\beta<\alpha+1$ is due to the fact that 
\begin{equation*}
   m^{\beta-\alpha} 
   \left\<\xi,
	 {\bf e}_k(\cdot)  
	 -
	 \sum_{x \in \bb T_m}{\bf e}_k(x)\1_{B_m(x)}(\cdot)
	 \right\>
\end{equation*}
has variance of order $\mc O\left(m^{2(\beta-\alpha-1)}\right)$ as $m \to \infty$ for each fixed $k$.
One can push an expansion past this restriction by renormalising the discrete fields. 
This renormalised version should incorporate Taylor polynomials of higher orders in the discretised noise.
After that, one can then follow our strategy to extend the proof.

In fact, assuming a full expansion  of the characteristic function $\phi_X$ (say for instance in the case of the simple random walk) would allow us to reiterate this argument and write a discrete fractional Gaussian field as a (possibly infinite) series of continuous Gaussian fields of decreasing regularity.
Again, this reminds us of ideas used in Hairer's decomposition in terms of regularity. 
Perhaps a easier analogy is a Taylor expansion: where each term describes variations of smaller amplitudes, but the term itself (the derivative in the case of Taylor expansion) becomes more irregular as we look at further terms into the expansion.

\subsection*{Global in time fluctuations}
One can easily improve the result in Theorem~\ref{thm:converge-of-fields-parab} to unbounded time intervals $[0,\infty)$ by introducing a weight density $\omega: [0,\infty) \to [0,\infty)$ satisfying $\int_{0}^{\infty} \omega(t) \max\{t^{2},t^3\}dt < \infty$ and looking at the space $L^2_\omega([0,\infty),H^{-s})$.
This decay is chosen to guarantee that the equivalent quantities to the norms of $A_m(t,k)$ and $B_m(t,k)$ (according to the notation of the proof of Theorem~\ref{thm:converge-of-fields-parab}) remain finite when we integrate $t$ over $[0,\infty)$. 

\subsection*{Fluctuations around solutions of non-linear SPDEs}

A interesting topic would be to study the fluctuations of discrete non-linear SPDEs.
A very informal idea is to use the da Prato-Debusche argument \cite{da2003strong}.
Consider the simple equation 
\begin{equation*}
  -(-\Delta)^{\alpha/2}X_\alpha
=
  \eta X^2_\alpha
  +\xi
\end{equation*}
in the torus, where $\xi$ is the space white-noise and $\eta \in  \bb R$. 
Depending on the value of $\alpha$, this equation is not well-posed, since we do not expect $X$ to be a function, but only a distribution.
However, we can try to write $X_\alpha = \Xi_\alpha + v_\alpha$ where $\Xi_\alpha$ is like in \eqref{eq:def-continuous-field}, as $\Xi_\alpha$ is dealing with the irregularity of the noise. We notice that $v_\alpha$ satisfies the formal equation
\begin{equation*}
  -(-\Delta)^{\alpha/2}v_\alpha
=
  \eta(\Xi_\alpha)^2
  +
  2\eta \Xi_\alpha v_\alpha
  +\eta v_\alpha^2
\end{equation*}
which we expect to be more regular as long as we manage to deal with the quadratic term, which can be given a formal meaning via Wick renormalisation.
And the equation has a meaning as long as $\eta$ is small enough and $\alpha$ is large enough to perform a fixed point argument.

One could follow the same idea for
\begin{equation*}
\mc L^m X_m
=
  \eta X_m^2
  +
  \xi^m,
\end{equation*}
where we take $\xi^m$ as i.i.d. random variables obtained by the same coupling we use later in this article, here we are taking $\mu_\alpha=1$ to simplify the exposition.
By decomposing $X_m = \Xi^m + v^m$, using Theorem~\ref{thm:converge-of-fields-elip}, we expect that we can also prove that $m^{\beta-\alpha}(v_m-v)$ converges to the solution of a PDE of the form
\begin{equation*}
  -(\Delta)^{\alpha/2} v_{Err}
  =
  \eta \Xi_\alpha  \cdot \Xi_{2\alpha-\beta}
  +
  (2 \eta v_\alpha + U + 2 \eta \Xi_\alpha )
  v_{Err},
\end{equation*}
after suitable renormalisation, where $U$ is an operator related to the functions $u_\beta$ showing in Theorem~\ref{prop:lclt}.
This equation is solvable via the da Prato-Debusche argument as long as $\alpha-\beta$ is sufficiently small.
As mentioned above, for higher values of $\alpha-\beta$, one needs to add derivatives of the white-noise, making the equation too irregular.
However, we expect this could still be solvable via more modern techniques of SPDEs.
Similar ideas hold for the parabolic version, and for higher dimensions.
We intend to study those results in future articles.

\section{Class of admissible and repairable distributions}\label{sec:example}

In this section, we will discuss a few examples and an explicit construction of admissible probability distributions with index $\alpha \in (0,2)$. 

\subsection{Basic properties of admissible distributions}\label{subsec:previously-known-admissible-distributions}

We start by stating simple properties of admissible distributions which will be useful the \textit{repairing process} of distributions.
\begin{lemma}\label{lem:closed-by-operations}
Let $p_{X_1}(\cdot)$ and $p_{X_2}(\cdot)$ be admissible distributions of independent random variables $X_1$ and $X_2$ with indexes $\alpha_1,\alpha_2\in (0,2]$, $\alpha_1\le \alpha_2$ and regularity sets $R_{\alpha_1}$, $R_{\alpha_2}$ respectively.
We have that their convolution,
\[
	p_X (x):=p_{X_1}*p_{X_2}(x)
\]
is admissible with index $\alpha_1$ and regularity set 
\[
	R^\prime_{\alpha_1} \subset (R_{\alpha_1}+R^+_{\alpha_2}) \cap(0,2+\alpha_1).
\]
Let  $\tilde{X}:=UX_1+(1-U)X_2$ where $U$ is a Bernoulli r.v.~with parameter $q \in [0,1]$, independent from $X_1$ and $X_2$, with distribution $p_{\tilde{X}}(\cdot)$.
We have that 
\begin{equation}\label{eq:alt-repair}
	p_{\tilde{X}}(x):=q\cdot p_{X_1}(x)+(1-q)p_{X_2}(x)
\end{equation}
is admissible with index $\alpha_1$ (for each $q$) and regularity set
\[
	R^*_\alpha \subset \Big (R_{\alpha_1} \cup R_{\alpha_2}^+ \Big)\cap(0,2+\alpha_1).
\]
\end{lemma}
\begin{proof}
This follows from the relations $\phi_X(\theta)=	\phi_{X_1}(\theta)\cdot	\phi_{X_2}(\theta)$ and $\phi_{\tilde{X}}(\theta)=	q\phi_{X_1}(\theta) + (1-q)	\phi_{X_2}(\theta)$.
\end{proof}

We can only describe the regularity sets as subsets since there might be cancellations due to the convolution or convex combinations.

\subsection{Old and new examples of admissible distributions}\label{subsec:explicit-examples-of-admissible-distributions}

Before we look for simpler examples, let us point out, that the result given by \cite{christoph1984asymptotic} in which they compute the characteristic function for certain random walks in the domain of attraction of a stable distribution, is covered by our class of admissible distributions.
\begin{proposition}\label{lem:Christoph}
Let $p$ be given by \eqref{eq:dist-from-integral}, then $p$ is admissible and its regularity set $R_\alpha$ satisfies $R_\alpha \subset \bb N \cap (0,2+\alpha)$ 
\end{proposition}
\begin{proof}
	This follows from \cite{christoph1984asymptotic} Example~2.15 and Theorem 2.22.
\end{proof}

The next result will be used in a constructive proof of the existence of admissible distributions in Proposition~\ref{prop:adm-are-enough}).
It is worth mentioning that we could simplify the proof if we were not interested in the signs of the constants, in particular of $\mu_2$. 

\begin{proposition}\label{prop:asymp-phi-alpha-asymmetric} 
Let $\alpha \in (0,2)\setminus\{1\}$ and define
\begin{equation}\label{eq:palpha-asymmetric} 
	p_{\alpha,+}(x) := 
	\frac{c^+_\alpha}{x^{1+\alpha}} \1_{\{x>0\}}
\end{equation}
where $c^{+}_\alpha=1/\zeta(1+\alpha) $ is the normalising constant and $\zeta$ the zeta-function.
Then the distribution $p_{\alpha,+}(\cdot)$ is
admissible with index $\alpha$ and the characteristic function for  the corresponding random variable $X$ satisfies
\begin{equation}\label{eq:char-alpha+}
	\phi_{\alpha,+}(\theta)	= 1- \mu_\alpha|\theta|^\alpha
	+i\mu^\prime_1\theta + \mu_2\theta^2 + i\mu^\prime_3\theta^3 +\mc O(|\theta|^{2+\alpha}),
\end{equation}
where $\mu^\prime_1,\mu_2,\mu^\prime_3 \in \bb R$. 
\end{proposition}

One can easily see that the asymmetric distribution for $\alpha=1$ is not
admissible. This is because the characteristic function $\phi_{1,+}(\cdot)$ satisfies
\begin{align*}
	\phi_{1,+}(\theta)
:=
	1-\mu_1|\theta|
	+i{\mu}^\prime_{1}|\theta|\log|\theta|
	+o(|\theta|\log|\theta|) \text{ as } \theta \longrightarrow 0,
\end{align*}
for some real constants $\mu_1,{\mu}^\prime_1$.
In the symmetric case, the $\log$-term will be purely imaginary, and it disappears when summed with its complex conjugate.
Indeed, the following proposition treats the symmetrical counterpart to Proposition~\ref{prop:asymp-phi-alpha-asymmetric}. 

\begin{proposition}\label{prop:asymp-phi-alpha}
The distribution $p_\alpha$ given in \eqref{def:palpha} for $\alpha \in (0,2)$ is admissible with index $\alpha$ and locally repairable.
Let $\phi_\alpha(\theta)$ be given by 
\begin{equation}\label{eq:char-function}
	\phi_\alpha(\theta)
	=
	c_\alpha \sum_{x \in \bb Z\backslash{\{0\}} }  \frac{ e^{i x\theta}}{|x|^{1+\alpha}}.
\end{equation}
\begin{enumerate}
\item Let $\alpha \neq 1$,  then $\phi_\alpha$ satisfies
\[
  \phi_{\alpha}(\theta) = 
  1 - \mu_{\alpha}|\theta|^{\alpha} + \mu_2 |\theta|^2 
  + \mc O(|\theta|^{2+\alpha}) \text{ as } |\theta| \rightarrow 0
\]
with coefficients $\mu_{\alpha}, \mu_2$ given by
\[
  \mu_{\alpha} = -2 c_{\alpha} \cos \left ( \frac{\pi \alpha}{2}\right ) \Gamma(-\alpha)
\quad \text{ and }\quad
    \mu_2
    >0.
\]
\item In the case $\alpha=1$, we have that 
\[
    \phi_{1}(\theta) = 
    1 - \frac{3}{\pi}|\theta| + \frac{3}{2\pi^2}|\theta|^2 + 
    \mc O(|\theta|^{3}) \text{ as } |\theta| \rightarrow 0.
\] %
\end{enumerate}
\end{proposition}

We will first prove Proposition~\ref{prop:asymp-phi-alpha-asymmetric} to explain our strategy to estimate such functions.

\begin{proof}[Proof of Proposition~\ref{prop:asymp-phi-alpha-asymmetric}] 
To prove this statement, we will use the Euler-Maclaurin formula \cite{Apo}, which states that for a given smooth function
$f \in C^{\infty}(\bb R)$, we have that 
\begin{equation}\label{eq:euler-mac}
	\sum_{x=1}^M f(x) - \int_{1}^M f(x)dx =
	\frac{f(1)+f(M)}{2} + R_{\alpha, +}^M,
\end{equation}
where $M \in \bb N$ and the remainder term $R^M_{\alpha,+}$  can be computed explicitly by
\[
	R_{\alpha, +}^M = \int_1^M f^{\prime}(z)P_1(z)dz,
\]
and $P_1(x)=B_1(x-\lfloor x\rfloor)$ with $B_1(\cdot)$ being the first periodised Bernoulli function, that is: $P_1(x)=\left(x - \lfloor x \rfloor \right)- \frac{1}{2}$.
We will apply this formula to the function $f(x) = \frac{1-\exp(i\theta x)}{|x|^{1+\alpha}}$.
Using that $\phi(-\theta)=\overline{\phi(\theta)}$, we  can assume that $\theta >0$.
As we take $M \to \infty$, the  left-hand side of \eqref{eq:euler-mac} becomes
\begin{equation}\label{eq:lhs-euler-mac}
	\frac{1-\phi_{\alpha,+}(\theta)}{c_{\alpha,+}} -
	\int_{1}^\infty \frac{1-\exp (i\theta z)}{z^{1+\alpha}}dz,
\end{equation}
where $c^{+}_\alpha$ is the normalising constant used in the definition of $p^{+}_\alpha(\cdot)$.
By a change of variables $z=x \theta$ in the above integral, we get
\begin{equation}\label{eq:integral-alpha-plus}
	\frac{1-\phi_{\alpha,+}(\theta)}{2c_\alpha} -
	\theta^\alpha
	\int_{\theta}^\infty \frac{1-\exp(iz)}{z^{1+\alpha}}dz.
\end{equation}
To analyse the integral in a systematic manner, we add and subtract counter terms at the singularity $0$ (notice that this is only necessary for $\alpha>1$)
\begin{align*}
	\theta^\alpha\int_{\theta}^\infty \frac{1-\exp(iz)}{z^{1+\alpha}}dz
& =
	\theta^\alpha\int_{1}^\infty \frac{1-\exp(iz)}{z^{1+\alpha}}dz
	+
	\theta^\alpha\int_{\theta}^1 \frac{1+iz-\exp(iz)}{z^{1+\alpha}}dz
	-
	i\frac{\theta-|\theta|^\alpha}{1-\alpha} .
\end{align*}
Using a Taylor expansion of $e^{iz}$ around $z=0$, we get
\begin{align*}
	\theta^\alpha
	\int_{\theta}^\infty \frac{1-\exp(iz)}{z^{1+\alpha}}dz
&=
	C^{+}_\alpha|\theta|^\alpha
	-
	\sum_{k=1}^3 \frac{(i\theta)^k}{k!(k-\alpha)} + \mc O(\theta^4),
\end{align*}
where
\begin{align*}
	C^{+}_\alpha
&=
	- \cos\left(\frac{\pi \alpha}{2}\right)\Gamma(-\alpha)
	\left(1
	+
	i \tan\left(\frac{\pi \alpha}{2}\right) \right),
\end{align*}
the constant was evaluated by means of Mellin transform (and analytic extension for the case $\alpha >1$), finally set $\mu_{\alpha}=2c_{\alpha}C^+_{\alpha}$.

Now, we turn to the right-hand side of \eqref{eq:euler-mac}. 
Note that $f(M) \rightarrow 0$ as $M\rightarrow \infty$.
Hence
\begin{align}
	\lim_{M\rightarrow \infty} 
	\frac{f(1)+f(M)}{2} + R^M_{\alpha}
& =
	\frac{1}{2}(1-\exp(i\theta))
	+ R^\infty_{\alpha}(\theta) \nonumber
\\ &=
	-\sum_{k=1}^3 \frac{(iz)^k}{2(k!)}
	+ R^\infty_\alpha
	+\mc O(\theta^4) \label{eq:prop5.4}
\end{align}
where
\begin{align}\label{def:r-infity+}
	R^\infty_{\alpha} = -\theta^{1+\alpha}
	\int_{\theta}^\infty
	\left(\frac{iz \exp(iz)+ (1+\alpha)(1-\exp(iz) )}{z^{2+\alpha}}\right)
	P_1\Big(\frac{z}{\theta}\Big)	dz.
\end{align}
The proof now follows from Lemma~\ref{lem:ap-R-alpha+} and identifying the fictional moments from collecting the corresponding coefficients from \eqref{eq:prop5.4}. 
\end{proof}

\begin{proof}[Proof of Proposition~\ref{prop:asymp-phi-alpha} ]
We start analogously as in the previous proof.
By using \eqref{eq:euler-mac} for $f(x)=(1-\cos(\theta x))x^{-1-\alpha}$ and taking $M \to \infty$, we get
\begin{equation}\label{def:r-infity} 
	\frac{1-\phi_{\alpha}(\theta)}{c_{\alpha}} -
	\int_{1}^\infty \frac{1-\cos (\theta z)}{z^{1+\alpha}}dz
	=
	\frac{1}{2}(1-\cos(\theta))
	+
	R^\infty_\alpha
\end{equation}
where $R^\infty_\alpha$ is the Euler-Maclaurin error.
In Lemma~\ref{lem:ap-R-alpha} we will estimate this integral in more detail, there we obtain
\begin{align*}
	R^\infty_\alpha(\theta): = K_2\theta^2 
	+\mc O(\theta^{2+\alpha}),
\end{align*}
where $K_2$ is a constant depending on $\alpha$ which is defined in \eqref{eq:asymp-of-R-2}.
By setting
\begin{equation}\label{eq:def-mu2}
	\mu_2
	=
	2c_\alpha 
	\left(
		\frac{1}{2(2-\alpha)} - \frac{1}{4} - K_2
	\right)
\end{equation}
 the statement for $\alpha \neq 1$ follows from applying Lemma~\ref{lem:Right-sign}. 

For the  case $\alpha=1$ the analysis becomes much simpler.  This is because the
first order term in Equation \eqref{asymp-of-R} vanishes. Since $\alpha=1$, the terms
$\theta^{1+\alpha}$ and $\theta^2$ collapse to the same term. The normalization
constant is equal to $c_1 = \frac{1}{2\zeta(2)} = \frac{3}{\pi^2}$.

Again, using Euler-Maclaurin we get that, for $\theta >0$
\begin{equation}\label{eq:euler-mac-alpha-1}
	\frac{1-\phi_1(\theta)}{2 c _1}
		 -
	\int_{1}^\infty \frac{1-\cos(\theta x)}{x^{2}}dx
 =
	\frac{1-\cos(\theta)}{2}
	+R^\infty_1,
\end{equation}
where the remainder term will be of order
\[
	R^\infty_1 =\int_1^\infty \left( \frac{1-\cos(\theta \cdot )}{(\cdot)^{2}} \right)^\prime (x) P_p(x)dx = \mathcal{O}(\theta^3).
\]

Since
\[
	\int_{0}^\infty \frac{1-\cos(z)}{z^{2}}dz
	=
	\frac{\pi}{2},
\]
we can write
\begin{align*}
	\theta \int_\theta^\infty \frac{ 1-\cos(z) }{z^{2}}dz
	&=
	\theta \int_0^\infty \frac{1-\cos(z)}{z^{2}}dz
	-
	\theta \int_0^\theta \frac{1-\cos(z)}{z^{2}}dz
	\\ &=
	\frac{\pi}{2}\theta
	-
	\frac{1}{2}\theta^2 + \mc O (\theta^4)
\end{align*}
where in the last line we used a simple Taylor expansion.
Collecting all coefficients corresponding to the powers of $\theta$ we obtain the result.
\end{proof}

\subsection{A criterium for admissibility}\label{subsec:a-criteria-for-admissibility}

It is natural to wonder whether our techniques can be applied to examples for which the limits $\lim_{x \to \pm\infty} |x|^{1+\alpha}\cdot p_X(x)$ are not well-defined.
A very natural criterium comes from tail bounds on the cumulative distribution function,
\begin{equation}\label{eq:general-condition-1}
  1- F_{X}(x) \stackrel{x \to \infty}{=} 
  \frac{c_{+}}{|x|^{\alpha}}
  +
  \sum_{\beta \in S_\alpha}\frac{c_{\beta,+}}{|x|^{\beta}} +
  o\left(\frac{1}{|x|^{4}}\right)
\end{equation}
and
\begin{equation}\label{eq:general-condition-2}
  F_{X}(x) \stackrel{x \to -\infty}{=} \frac{c_{-}}{|x|^{\alpha}} +
  \sum_{\beta \in S_\alpha}\frac{c_{\beta,-}}{|x|^{\beta}}
  +
  o\left(\frac{1}{|x|^{4}}\right)
\end{equation}
for a finite $S_{\alpha} \subset (0,2+\alpha)$ and  two positive constants $c_+$ and $c_-$ and any arbitrary real constants 
$c_{\beta,+},c_{\beta,-}, \beta \in S_\alpha$. 

\begin{lemma}\label{lem:general-tail}
  If $\alpha \in (1,2)$, let $p_X(\cdot)$ be the distribution given by 
  \begin{equation*}
    p_X(x):=F_X(x)-F_X(x-1)
  \end{equation*}
  where $F_X(\cdot)$ satisfies estimate~\eqref{eq:general-condition-1}. Then we have that $p_X(\cdot)$ is admissible with regularity set $R_\alpha \subset (\bb Z \cup \{1+\alpha\}) \cap (0,2+\alpha).$ 
\end{lemma}

\begin{proof}
	We will concentrate on the case $\tilde{p}_\alpha(x)= \tilde{p}_\alpha(|x|)=\frac{1}{2}\frac{1}{|x|^{\alpha}}-\frac{1}{(|x|-1)^{\alpha}}$ for $x \not \in \{-1,0,1\}$ , which means that $c_+=c_-=1$ and that the error term is zero.
The non-symmetric case can be dealt with in a very similar manner.
We will comment on the presence of an error term in later.

The proof is essentially the same as in Proposition~\ref{prop:asymp-phi-alpha} plus an summation by parts.
Indeed, remember that for two sequence  $\{f_k\}_{k \ge 1}$ and $\{g_k\}_{k \ge 1}$, we have that 
for any $M \in \bb N$
\begin{equation}\label{eq:summation-by-parts}
	\sum_{k = 1}^M
	f_k [g_{k+1}-g_k] 
	= 
	f_M g_{M+1}- f_1 g_1
	-
	\sum_{k = 1}^M
	g_{k+1} [f_{k+1} - f_k].
\end{equation}
By taking $f_k=(1-\cos(\theta k))$, $g_{k}=F_X(k-1)-1$ and $M \to \infty$ , we get
\begin{align*}
	2(1-\tilde{\phi}_\alpha(\theta))
& =
	\sum_{x \in \bb N} (1-\cos(\theta x))\tilde{p}_\alpha(x)
\\&  = 
  (1-\cos(\theta))(1-F_X(1))
  +
  \sum_{x \in \bb N} [\cos(\theta(x+1))-\cos(\theta x)] (1-F_X(x))
\\&  = 
  (1-\cos(\theta))(1-F_X(1))
  +
  (\cos(\theta)-1)
  \sum_{x \in \bb N} \frac{\cos(\theta x)}{x^\alpha}
  +
  \sin(\theta) 
  \sum_{x \in \bb N}  \frac{\sin(\theta x)}{x^\alpha}.
\end{align*}
The proof now follows from analysing each of the infinite sums by using Euler-Maclaurin, just as before.
Notice that the leading term, as $\theta \to 0$  for both series is of  order $\mc O(|\theta|^{\alpha-1})$, but the terms $(1-\cos(\theta))$ and $\sin(\theta)$ help to recover the original rate.
The condition $\alpha>1$ is required in order to have the functions $\theta \mapsto \int_{[\theta,\infty)} \frac{\cos(\theta x)}{|x|^{\alpha}}dx $ and $\theta \mapsto \int_{[\theta,\infty)}  \frac{\sin(\theta x)}{|x|^{\alpha}} dx$ to be well-defined.

The error bound can be dealt with by defining 
\begin{align*}
	\phi_{E_\alpha}(\theta)
&  = 
  (1-\cos(\theta))J(1)
  +
  (\cos(\theta)-1)
  \sum_{x \in \bb N} \cos(\theta x)J(x)
  +
  \sin(\theta) 
  \sum_{x \in \bb N}  \sin(\theta x)J(x),
\end{align*}
where $J(x):=(1-F_X(x))-|x|^{-\alpha} = c_{\pm,\delta} |x|^{-\delta} + o(|x|^{-\delta}) $ for some $\delta \in (\alpha,4]$ as $|x| \to \infty$.
We can then analyse $\phi_{E_\alpha}$ using analogous estimates as we used for the first part of this proof but with $\delta$ instead of $\alpha$, notice that as we are not assuming the positiveness of the constants $c_{+,\beta}, c_{-,\beta}$, $\phi_{E_\alpha}$ may not be a characteristic function of a random variable, however, this is irrelevant for this proof.
We iterate this argument until we have exhausted the set $S_{\alpha}$, leaving us with an error of order 
\begin{align*}
	\phi_{E^*}(\theta)
&  = 
  (1-\cos(\theta))J^* (1)
  +
  (\cos(\theta)-1)
  \sum_{x \in \bb N} \cos(\theta x)J^*(x)
  +
  \sin(\theta) 
  \sum_{x \in \bb N}  \sin(\theta x)J^*(x),
\end{align*}
where $|J^*(x)| = o(|x|^{-4})$, at which point we can just differentiate the expression above in order to obtain the desired bounds. 
\end{proof}

\begin{remark}\label{rem:dom-attraction}
Remember that a distribution is in the domain of  attraction of an $\alpha$-stable distribution if, and only if,
\begin{equation*}
 \lim_{x\to \infty} (1- F_{X}(x))|x|^{\alpha} = c_{+} \in[0,\infty)
\end{equation*}
and
\begin{equation*}
 \lim_{x\to -\infty} F_{X}(x)|x|^{\alpha} = c_{-} \in [0,\infty)
\end{equation*}
with $c_++c_- > 0$, see e.g. \cite{Taqqu}.
Therefore, the condition given by Lemma~\ref{lem:general-tail} is a natural quantitative version of this characterisation.
\end{remark} 

The reason why we concentrate on the case where the density decays like a power law, (rather than the complement of the cumulative density) is that the former can be used for any $\alpha$, and also for any dimension - see Section~\ref{sec:general}. 

\subsection{Construction of admissible distributions in the domain of attraction of any given stable distribution}
\label{subsec:construction-of-admissible-distributions-in-the-domain-of-attraction-of-any-given-stable-distribution}


We end this section by giving a constructive proof of the existence of an admissible distribution in the domain of normal attraction of any given stable distribution (at least for $\alpha \neq 1$).
In fact, such a construction will be derived from power-law distributions.
It shows that the class of examples provided in this article is large. 
\begin{proposition}\label{prop:adm-are-enough}
Let $\bar{X}$ be a stable random variable with parameters $(\alpha,\beta,\gamma,\varrho)$ as in Definition \ref{def:stable}. 
%
There exists an distribution $p_X \in \mathcal{A}$ such that for the corresponding characteristic functions we have
\begin{equation}\label{eq:adm-are-enough} 
	\phi_{\alpha,\beta,\gamma,\varrho}(\theta)
	=
	\phi_{X}(\theta)
	+ o\left(|\theta|^\alpha\right).
\end{equation}

The same holds for $\alpha=1, \beta = 0$ and $\gamma \in \bb R^{+},\varrho \in \bb R$. 
\end{proposition}

\begin{proof}
The construction is simple and follows by  modifying the characteristic function in order to add one parameter, i.e.  $\beta,\gamma,\rho$, at the time. 
We will restrict ourselves to the case $\alpha \neq 1$ as it is more elaborate.

We start by defining a distribution $p_{W_1}(\cdot)$ with the correct parameters $\alpha$ and $\beta$, this can be done by choosing an appropriate convex sum of totally asymmetric random variables.
Let  $p_{\alpha,+}(\cdot)$ be the probability distribution given in Proposition~\ref{prop:asymp-phi-alpha-asymmetric} and 
define $p_{\alpha, -}(\cdot):= p_{\alpha,+}(-\cdot)$. 
Now, let $q_1= \frac{\beta+1}{2}$ and note that $q_1\in [0,1]$. 
The distribution
\begin{equation*}
	p_{W_1}:= q_1 \cdot p_{\alpha,+} + (1-q_1)\cdot p_{\alpha,-}
\end{equation*}
belongs to $\mc A$ due to Lemma~\ref{lem:closed-by-operations}.
Moreover,
\begin{equation*}
	\log(\phi_{W_1}(\theta))
	= 
	i \beta \theta -|(\mu_\alpha)^{1/\alpha} \theta|^{\alpha}\left(1-i \beta \sgn(\theta) \tan\left(\frac{\pi\alpha}{2}\right)\right) 
	+ o(|\theta|^\alpha),
\end{equation*}
due to Proposition~\ref{prop:asymp-phi-alpha-asymmetric}. 

Now we will define an admissible distribution $p_{W_2}$ with correct parameters $\alpha,\beta$ and $\gamma$.
This will be done in two steps.
First, we define $M:= \left\lceil \frac{\gamma}{(\mu_\alpha)^{1/\alpha}}\right\rceil \ge 1$ and consider 
\begin{equation*}
	p_{W^\prime_2}  
	:=
	\underbrace{p_{W_1}
		*
		\dots
		*
		p_{W_1}
	}_{M \text{ times }}.
\end{equation*}
We have $p_{W^\prime_2} \in \mc A$  and
\begin{align*}
	\log(\phi_{W^\prime_2}(\theta))
	= 
	i M \beta \mu'_1\theta -|M\cdot(\mu_\alpha)^{1/\alpha} \theta|^{\alpha}\left(1-i \beta \sgn(\theta) \tan\left(\frac{\pi\alpha}{2}\right)\right) 
	+ o(|\theta|^\alpha).
\end{align*}
Now, define	$q_2:= \frac{\gamma}{
M \cdot(\mu_\alpha)^{1/\alpha}}\in [0,1]$.
The distribution
\begin{equation*}
	p_{W_2}(x):= q_2 \cdot p_{W^\prime_2}(x)+ (1-q_2)\cdot\delta_0(x)
\end{equation*}
is admissible and satisfies
\begin{equation*}
	\log (\phi_{W_2}(\theta))
	:= 
	i \gamma^\alpha \mu_1'\beta\theta -| \gamma\theta|^{\alpha}\left(1-i \beta \sgn(\theta) \tan\left(\frac{\pi\alpha}{2}\right)\right) 
	+ o(|\theta|^\alpha),
\end{equation*}
thanks to the asymptotics of $\phi_{W_2^\prime}$ around $0$ and the Taylor expansion of the functions $z \mapsto e^z, z \mapsto \log(1+z)$ around $z=0$.   

Finally, we recover the drift parameter $\varrho$ by setting $D = \left\lceil \varrho-  q_2 M \beta \mu_1'\right\rceil$ and $q_3 \in [0,1]$ such that
\begin{align*}
	q_3 (D-1) + (1-q_3)D= \varrho-  q_2 M \beta \mu_1'.
\end{align*}
Then, it is easy to show that 
\begin{equation*}
	p_{X}:=  p_{W_2}*(q_3\cdot\delta_{D-1}+ (1-q_3)\cdot\delta_{D}) 
\end{equation*}
belongs to $\mc A$ and satisfies the relation \eqref{eq:adm-are-enough}.
\end{proof}

\section{Proofs of Local Central Limit Theorems}\label{sec:lclt}

In this section we will prove Theorems \ref{thm:lclt} and \ref{prop:lclt}. 

\begin{proof}[Proof of Theorem \ref{thm:lclt}]
We will prove cases (i) and (iii) since case (ii) is a corollary of case (i).\\
\noindent
\textbf{Case (i): $p_X(\cdot)$ asymmetric}\\
Consider $(X_i)_{i\in \mathbb{N}}$ a sequence of admissible random variables and $(\bar{X}_i)_{i\in \mathbb{N}}$ a sequence of i.i.d.  $\alpha$-stable random variables  with laws $p_X(\cdot)$ resp.~$p_{\bar{X}}(\cdot)$, where $p_X$ is in the domain of normal attraction of $p_{\bar{X}}$. 
Let $S_n =\sum_{i=1}^n X_i$ resp.~$\bar{S}_n=\sum_{i=1}^n \bar{X}_i$ with probability distributions denoted by $p^n_X(\cdot)$ resp.~$p^n_{\bar{X}}(\cdot)$.
 We want to compare the probability distributions $p^n_{X}(\cdot)$ and $p^n_{\bar{X}}(\cdot)$ using their representation in terms of inverse Fourier transforms.
More precisely we have that
	\[
		p^n_{X}(x)
		=
		\frac{1}{2 \pi}
		\int_{-\pi}^\pi \phi^n_X (\theta) e^{-i x \theta} d\theta
	\]
resp.
\[
	p^n_{\bar{X}}(x)
	=
\frac{1}{2\pi}	\int_{-\infty}^\infty
	\phi_{\bar{X}}(\theta) e^{-i \theta \cdot x} d\theta.
\]

Using a change of variable formula, we get
\[
      p^n_X(x)
      =
      \frac{1}{2 \pi n ^{1/\alpha}}
      \int_{-\pi n^{1/\alpha}}^{\pi n^{1/\alpha}}
      \phi^n_X \Bigg( \frac{\theta}{n^{1/\alpha}} \Bigg)
      e^{ -i x\frac{\theta}{ n ^{1/\alpha}}} d\theta.
\]
Given $\varepsilon>0$, notice that $\sup_{ \theta \in \bb T\setminus
[-\varepsilon,\varepsilon]}|\phi_X (\theta)|<1$, as $X$ is 
supported in $\bb Z$, see \cite[Lemma 2.3.2]{Limic10}.
To get
\[
	p^n_X(x)
	=
	\frac{1}{2 \pi n ^{1/\alpha}}
	\int_{-\varepsilon n^{1/\alpha}}^{\varepsilon n^{1/\alpha}}
	\phi^n_X \Bigg( \frac{\theta}{n^{1/\alpha}} \Bigg)
	e^{ - i x \frac{ \theta}{ n ^{1/\alpha}}} d\theta
	+ \mc O (e^{-cn})
\]
for some positive constant $c>0$. Analogously, we have that
\begin{align*}
	{p}^n_{\bar{X}}(x)
& = 	\frac{1}{ 2 \pi n^{1/\alpha}}
	\int_{-\varepsilon n^{1/\alpha}}^{\varepsilon n^{1/\alpha}} 
	\phi^n_{\bar{X}}\left(\frac{\theta}{n^{1/\alpha}}\right)
	e^{- i x \frac{ \theta}{n^{1/\alpha}}} d\theta 
 +
	\frac{1}{ 2 \pi n^{1/\alpha}}
	\int_{|\theta|> \varepsilon n^{1/\alpha}}
	\phi^n_{\bar{X}}\left(\frac{\theta}{n^{1/\alpha}}\right)
	e^{- i x \frac{\theta}{n^{1/\alpha}}} d\theta.
\end{align*}

One can easily check that,
\[
	\int_{|\theta|> \varepsilon n^{1/\alpha}}
	\phi^n_{\bar{X}}\left(\frac{\theta}{n^{1/\alpha}}\right)
	e^{-\frac{i x \theta}{n^{1/\alpha}}} d\theta
	= \mc O \left(e^{-c' n}\right),
\]
for some constant $c'>0$.
Write $\phi^n_X \Big(\frac{\theta}{n^{1/\alpha}}\Big) = [1+F_n(\theta)] \phi_{\bar{X}}(\theta)$,
we can concentrate our efforts
into bounding
\begin{align}
	\label{thm:lclt-int-to-bound}
	\int_{-\varepsilon n^{1/\alpha}}^{\varepsilon n^{1/\alpha}}
	F_n(\theta)
	\phi^n_{\bar{X}}\left(\frac{\theta}{n^{1/\alpha}}\right)
	e^{-\frac{i x \theta}{n^{1/\alpha}}} d\theta.
\end{align}
Remember the notation $\beta_1:=\min(J_\alpha)$ where $J_\alpha := \spann (R_\alpha)$.
Now, we use
\[
    |F_n(\theta)| 
	\lesssim n^{1-\frac{\beta_1}{\alpha}}|\theta|^{\beta_1}
\]
for $|\theta| < \varepsilon n^{1/\alpha}$
(possibly for smaller value of $\varepsilon$).
With this, we get
\[
\begin{split}
\left| p^n_X(x) - p^n_{\bar{X}}(x) \right| & =
	\Big |\frac{1}{ 2\pi n^{1/\alpha}}\int_{|\theta|<\varepsilon n^{1/\alpha}}
	\phi^n_{\bar{X}}\left(\frac{\theta}{n^{1/\alpha}}\right)
	F_n(\theta)d\theta\Big| + \mc O(e^{-c^\prime n})
  \\ & \lesssim
	\frac{1}{n^{\frac{\beta_1+1-\alpha}{\alpha}}} \underbrace{\int_{|\theta|<\varepsilon n^{1/\alpha}}
	 e^{- \mu_\alpha |\theta|^\alpha}
	 |\theta|^{\beta_1}
	 d\theta}_{\mc O (1)}
	+ \mc O(e^{-c^\prime n})
\end{split}
\]
and that the integral on the r.h.s.~is bounded as $n \longrightarrow \infty$.

\noindent
\textbf{Case (ii): $p_X(\cdot)$ asymptotically repairable}\\
We will prove the statement in a similar manner, so
we will only highlight the main differences. Write
\begin{align*}
	p^n_{\bar{X}+ \bar{Z}}(x)
&=
	\frac{1}{2 \pi}
	\int_{-\infty}^\infty e^{-n \mu_\alpha |\theta|^\alpha-n \mu_2 |\theta|^2}
	e^{-i x \theta} d\theta
\\&=
	\frac{1}{2 \pi n^{1/\alpha}}
	\int_{-\infty}^\infty e^{- \mu_\alpha |\theta|^\alpha-n^{(1-2/\alpha)} \mu_2 |\theta|^2}
	e^{-\frac{i x \theta}{n^{1/\alpha}} } d\theta
\end{align*}
and write $\phi^n_X \Big(\frac{\theta}{n^{1/\alpha}}\Big) =
[1+F_n(\theta)]\exp\left( - \mu_\alpha |\theta|^\alpha-n^{(1-2/\alpha)} \mu_2
|\theta|^2\right)$. Notice that $1-\frac{2}{\alpha}<0$.

One can easily check that,
\[
	\int_{|\theta|> \varepsilon n^{ 1/\alpha}}
	e^{- \mu_\alpha |\theta|^\alpha-n^{1- \frac{2}{\alpha}} \mu_2 |\theta|^2}
	e^{- i x \frac{\theta}{ n^{1/\alpha}}} d\theta
	= \mc O (e^{-cn}),
\]
for some constant $c>0$. The statement will follow once we bound 
\[
	\int_{-\varepsilon n^{1/\alpha}}^{\varepsilon n^{1/\alpha}}
	F_n(\theta)e^{- \mu_\alpha |\theta|^\alpha-n^{1- \frac{2}{\alpha}} \mu_2 |\theta|^2}
	e^{- i x \frac{\theta}{ n^{ 1/\alpha}}} d\theta \lesssim  n^{-1/\alpha}.
\]
Analogously to before, we have that for $|\theta|\le \varepsilon n^{1/\alpha}$, we have
\[
	|F_n(\theta)| \lesssim \frac{|\theta|^{2\alpha}}{n}.
\]
This concludes the  claim.
\end{proof}

We proceed with the proof of Theorem \ref{prop:lclt}. 
\begin{proof}[Proof of Theorem \ref{prop:lclt}]
In this case, we are only dealing with symmetric distributions.
Using similar ideas as before in the proof of Theorem \ref{thm:lclt}, assume
again that $\theta >0$, we write
	\[
		p^n_X(x)
		=
		\frac{1}{2 \pi n ^{1/\alpha}}
		\int_{-\varepsilon n^{1/\alpha}}^{\varepsilon n^{1/\alpha}}
		[1+F_n(\theta)]e^{- \mu_\alpha |\theta|^\alpha}
		e^{ - i  x \theta n ^{-\frac{1}{\alpha}}} d\theta
		+ \mc O (e^{-cn^{1/\alpha}})
	\]
for some positive constant $c >0$. We have that 
\begin{equation}\label{eq:exp-Fn}
	F_n(\theta)
	=
	\sum_{\beta \in J_\alpha} C_\beta \frac{n}{n^{\beta/\alpha}}|\theta|^\beta + 
	\sum_{\beta \in J_\alpha} C^\prime_\beta \frac{n}{n^{\beta/\alpha}}\sgn(\theta)|\theta|^\beta + 
	\mc O \left( \frac{|\theta|^{2+\alpha}}{n^{2/\alpha}} \right),
\end{equation}
where we used the Taylor polynomial of 
\begin{align*}
  t \mapsto e^{ \sum_{\beta \in R_{\alpha}} n^{1-\beta/\alpha} \mu_{\beta} |t|^{\beta}
+
 \sum_{\beta \in R_{\alpha}} n^{1-\beta/\alpha} \mu^\prime_{\beta} \sgn(t)|t|^{\beta}
  }
\end{align*}
truncated at level
$\mc O \left(\frac{t^{2+\alpha}}{n^{2/\alpha}}\right)$.

Define
	\begin{equation*}
		u_\beta(x):=
		\frac{1}{2\pi}
		\int_{\bb R} 
			 |\theta|^{\beta} \phi_{\bar{X}}(\theta)
			 e^{ix \theta}
		d\theta,
		\text{ and }
		u^\prime_\beta(x):=
		\frac{1}{2\pi}
		\int_{\bb R} 
			 \sgn(\theta) |\theta|^{\beta} \phi_{\bar{X}}(\theta)
			 e^{ix \theta}
		d\theta,
	\end{equation*}
hence we have that for $|\theta| < \varepsilon n^{1/\alpha}$
\begin{align*}
	\Bigg|
	&
	p^n_X(x)- p^n_{\bar{X}}(x)-
	\sum_{\beta\in J_\alpha} C_\beta
	\frac{u_\beta\left( \frac{x}{n^{1/\alpha}} \right)}{n^{(1+\beta-\alpha)/\alpha}}
	 \Bigg|
	 \\ & \lesssim
	  \int_{-\varepsilon n^{1/\alpha}}^{\varepsilon n^{1/\alpha}}
	  	\frac{
		 |\theta|^{2+\alpha} e^{-\mu_\alpha |\theta|^\alpha}}{ n^{3/\alpha}}
		d\theta
	  +
		\sum_{\beta \in J_\alpha} (|C_\beta|+|C^\prime_\beta|)
		\int_{ \bb R \setminus [-\varepsilon n ^{1/\alpha},\varepsilon n^{1/\alpha}]}
		\frac{ |\theta|^{\beta} e^{-\mu_\alpha |\theta|^\alpha}}{ n^{(1+\beta-\alpha)/\alpha}}
		d\theta,
\\ & \lesssim
	n^{-3/ \alpha} 
	  +
	  \mc O \left( e^{-c n} \right)
\end{align*}
for some $c>0$ and $n$ large enough.
\end{proof}
\section{Proofs for Green function/potential kernel expansion}\label{sec:Green}

In this section we will develop potential kernel estimates stated in Theorems
\ref{thm:Green} and \ref{thm:green-1}.  The strategy will be to use detailed
knowledge of the expansion $\phi_X(\cdot)$ and not the LCLT theorem as was done
for the equivalent problem in the classical case in \cite{Limic10}.
Again, here we are restricting ourselves to the symmetric case.

\begin{proof}[Proof of Theorem \ref{thm:Green}]
\textbf{Case (i) $p_X(\cdot)$ repaired} \\
Let
us evaluate the expression 
\[
	a_{X}(x) = \frac{1}{2\pi} \int_{-\pi}^\pi \frac{1}{1-\phi_X(\theta)}
	(\cos(\theta x)-1)d\theta.
\]

The idea is to compare $a_X(x)$ with the potential kernel $a_{\bar{X}}(\cdot)$ of a symmetric stable process $(\bar{X}_t)_{t\geq 0}$ with multiplicative constant $\mu_{\alpha}$ whose characteristic function is given by $\phi_{\bar{X}_t}(\theta) = e^{-\mu_\alpha t |\theta|^\alpha}$.
 This is more convenient since it can be explicitly computed.
 Using that $(t,\theta) \mapsto e^{-\mu_\alpha t |\theta|^\alpha}(\cos(\theta x)-1)$ is in $L^1(\bb R_+ \times \bb R)$, we can use Fubini to get 

\begin{align*}
	a_{\bar{X}}(x)
	&=
	\frac{1}{2\pi} \int_{\bb R}\int_{0}^\infty e^{-t \mu_\alpha |\theta|^\alpha }dt
	(\cos (\theta x) - 1)  d\theta
	\\&=
	\left(
		\frac{1}{2\pi \mu_\alpha}\int_{\mathbb{R}} \frac{1}{ |\theta|^\alpha}
		(\cos (\theta) - 1)  d\theta
	\right)
	|x|^{\alpha-1}
\end{align*}
which gives the expression for the constant $C_\alpha$.
We write 
\[
	a_X(x) = a_{\bar{X}}(x)
	+ \underbrace{\left( a_{X}(x)  - \tilde{a}_{\bar{X}}(x) \right)}_A
	- \underbrace{\left( a_{\bar{X}}(x) - \tilde{a}_{\bar{X}}(x) \right)}_B,
\]
where
\[
	\tilde{a}_{\bar{X}}(x)
:=
	\frac{1}{2\pi \mu_\alpha} \int_{-\pi}^\pi\frac{1}{ |\theta|^\alpha}
	(\cos (\theta x) - 1)  d\theta.
\]

The reminder of the proof is divided  into two parts: estimating the term in $A$  by using H\"older continuity and then the term in $B$  by using an interplay of Fourier transform in the torus $\bb T$ and in $\bb R$ plus a trick involving dyadic partitions of the unity.

We start by analysing the term
\begin{align*}
	a_{X}(x)  - \tilde{a}_{\bar{X}}(x)
&=
	\frac{1}{2\pi} \int_{-\pi}^{\pi}
	\left( \frac{1}{ 1-\phi_X(\theta)}- \frac{1}{ \mu_\alpha |\theta|^\alpha}\right)
	(\cos (\theta x) - 1)  d\theta
	\\ &=
	\frac{1}{2\pi} \int_{-\pi}^{\pi}
	\frac{h_X(\theta)}{ \mu_\alpha |\theta|^\alpha (1- \phi_X(\theta))}
	(\cos (\theta x) - 1)  d\theta.
\end{align*}
where 
\begin{equation*}
	h_X(\theta) :=
	\phi_X(\theta)-\left(  1-\mu_\alpha |\theta|^\alpha \right)
	= \mc O (|\theta|^{2+\alpha})
\end{equation*}
since $p_X(\cdot)$ is repaired. 

It is important to notice that $h_X(\theta)$ is in $C^{1,\alpha-1-}(\bb T)$ due to Lemma \ref{lemma-app-phi-smooth} and the continuity of $\theta \mapsto 1-\mu_\alpha |\theta|^{\alpha}$.
Denote by $ \tilde{h}_X(\theta) := \frac{h_X(\theta)}{ \mu_\alpha |\theta|^\alpha (1- \phi_X(\theta))}$ which is in $L^1(\bb T)$, as $(1-\phi_X(\theta))\neq 0$ for all $\theta \in \bb T \setminus\left\{ 0 \right\} $ again due to the fact that $X$ is supported in $\bb Z$.

Hence, we write for $A$
\begin{align*}
	a_{X}(x)  - \tilde{a}_{\bar{X}}(x)
	&=
	-\frac{1}{2\pi} \int_{-\pi}^{\pi}
	\tilde{h}_X(\theta) d\theta
	+
	\underbrace{\frac{1}{2\pi} \int_{-\pi}^{\pi}
	\tilde{h}_X(\theta) \cos (\theta x)
	d\theta}_{I(x)}.
\end{align*}
The first integral in the r.h.s.~is finite and does not depend on $x$.
 We will show that the second integral on the r.h.s.~above is of order $\mathcal{O}(|x|^{ \frac{\alpha-2}{3+\varepsilon}})$.

This estimate is based on the fact that such integrals are Fourier coefficients of a function in $C^{0, \frac{2-\alpha}{3+\varepsilon}}(\bb T)$ for some $\varepsilon>0$ small enough.

We write 
\[
	f_1(\theta) :=
	\frac{h_X(\theta)}{|\theta|^{2\alpha}}
\]
and
\[
	f_2(\theta):=
	\frac{|\theta|^{\alpha}\left( \mu_\alpha |\theta|^\alpha - h_X(\theta) \right)}{|\theta|^{2\alpha}}=
	\mu_{\alpha} - \frac{h_X(\theta)}{|\theta|^{\alpha}}.
\]
Now, we use Lemma \ref{lem:app-holder-quocient}  to determine the degree of H\"older continuity of $f_1(\cdot)$ and $f_2(\cdot)$.
 For $f_1(\cdot)$ we can choose  $\beta = \alpha-1-\varepsilon$ for any $\varepsilon \in (0,\alpha-1)$, $\beta_0=2+\alpha$ and $\beta_1=2\alpha$ to obtain that $f_{1}(\cdot)$ is H\"older continuous with $\alpha_1 = \frac{2-\alpha}{3+\varepsilon}$ for $\alpha > 1$.
For $f_2(\cdot)$,  we can choose  $\beta = \alpha-1-\varepsilon$, $\beta_0=2+\alpha$ and $\beta_1=\alpha$ which yields to an order $\alpha_2=\frac{2}{3 +\varepsilon}$.
Since $f_2(\cdot)$ is bounded away from 0 we have that the reciprocal $1/f_2(\cdot)$ is H\"older continuous of order $\alpha_2$ as well.
 Therefore the product $f_1(\cdot) \cdot \frac{1}{f_2(\cdot)}$ is H\"older continuous of order $\alpha_1 \wedge \alpha_2=\alpha_1$.
 This implies that $I(x) = \mc{O} (|x|^{-\alpha_1})$, see \cite[Theorem 3.3.9]{grafakos2008classical}.

For the second part of the proof, we estimate the term $B=a_{\bar{X}}(x) - \tilde{a}_{\bar{X}}(x)$.
To do so, let $\varphi \in C^\infty(\bb{R})$ be a symmetric cutoff function such that $\varphi \equiv 1$ in $\bb R\setminus \left[-\pi+\eta,\pi-\eta \right]$ for some arbitrarily small $\eta>0$ and such that $\varphi \equiv 0$ in $\left[ -\pi+2\eta,\pi-2\eta \right]$, we now have

\begin{align*}
	 2 \pi\mu_\alpha
	\left[
		a_{\bar{X}}(x) - \tilde{a}_{\bar{X}}(x)
	\right]
	&=
	\int_{\bb R \setminus \bb T} \frac{1}{|\theta|^\alpha}
	(\cos (\theta x) - 1)  d\theta
\\ & =
	-\underbrace{\int_{\bb R \setminus [-\pi,\pi]} \frac{1}{|\theta|^\alpha}  d\theta}_{\frac{\pi^{1-\alpha}}{\alpha-1}}
	+
	\underbrace{\int_{\bb R} \varphi(\theta) \frac{1}{|\theta|^\alpha}
	\cos (\theta x) d\theta}_{J_1(x)} 
\\ & \phantom{=}
	+  \underbrace{\int_{\bb R} \left[ \1_{\{|\theta|>\pi\}}-\varphi(\theta) \right]  \frac{1}{|\theta|^\alpha}
	\cos (\theta x) d\theta}_{J_2(x)}.
\end{align*}
The constant $-\frac{\pi^{1-\alpha}}{2\pi \mu_{\alpha}(\alpha-1)}$ gives the
second contribution to $C_0$.  We write $J_1(x)= \mc F\left(
\frac{\varphi(\cdot)}{|\cdot|^\alpha} \right)(x)$.

In order to analyse $J_1(x)$ we need to use a dyadic partition of the unity to show that this term decays faster than any polynomial.
 Let $\psi_{-1},\psi_0$ be two radial functions such that $\psi_{-1} \in C_c^{\infty}(B_{\pi}(0))$ and $\psi_0 \in C_c^{\infty}(B_{2\pi}(0) \setminus B_{\pi}(0))$.
They satisfies 
\begin{equation}
	\label{eq:dy-part}
	1 \equiv \psi_{-1}(\theta)
	+ \sum_{j=0}^\infty \underbrace{\psi_0(2^{-j}\theta)}_{=:\psi_j(\theta)}.
\end{equation}
Such functions exist by Proposition 2.10 in \cite{Chemin}, it is an application
of Littlewood-Payley theory.
Define 
\[
	\varrho(\theta):=\frac{\varphi(\theta)}{|\theta|^\alpha}\psi_{-1}(\theta)
\qquad\text{and}\qquad
	\nu(\theta):=\frac{\varphi(\theta)}{|\theta|^\alpha}\psi_0(\theta)
	\equiv \frac{1}{|\theta|^\alpha}\psi_0(\theta),
\]
where, in the identity, we used that $\varphi \equiv 1$ in the $\supp (\psi_0)$.
We have that both $\varrho,\nu\in C^\infty_c (\bb R)$ and therefore their Fourier
transforms decay faster than any polynomial, that is, for any $N>1$, we have that
\begin{equation}
	\label{eq:fast-decay}
	\mc F(\nu)(x),\mc F(\varrho)(x) = \mc O(|x|^{-N}).
\end{equation}
 The fact that we can exchange the infinite sum with the 
Fourier transform is a result of the dominated convergence theorem.

Multiply both sides of \eqref{eq:dy-part} by 
$\varphi(\theta)/ |\theta|^{\alpha}$ , compute $\mc F$
and use the scaling property of the Fourier transform to get
\begin{equation}
	\label{eq:exp-J1}
	J_1(x) = \mc F (\varrho)(x)+ \sum_{j=0}^\infty 2^{(1-\alpha)j }\mc F (\nu)(2^jx).
\end{equation}
By using \eqref{eq:fast-decay} and \eqref{eq:exp-J1}, we get that $J_1(x)=\mc O(|x|^{-N})$
for all $N\ge 1$.
Finally we estimate $J_2(x)$
\begin{align*}
	J_2(x)
&= 
	\int_{-\pi}^\pi 
	\left[ \1_{[|\theta|>\pi]}-\varphi(\theta) \right]  \frac{1}{|\theta|^\alpha}
	\cos (\theta x) d\theta
\\&= 
	- \int_{-\pi}^\pi 
	 \varphi(\theta)  \frac{1}{|\theta|^\alpha}
	\cos (\theta x) d\theta
\end{align*}
where we used that $\varphi\equiv 1$ for $|x|>\pi$.
We can write $J_2(x)=\mc F_{\bb T}(g)(x)$.
Notice that $g$ is $C^{0,1}(\bb T)$, and therefore $J_2(x)$ decays as  $\mc O(|x|^{-1})$ which is faster than $\mc O(|x|^{\frac{\alpha-2}{3+\varepsilon}})$ because $\alpha\in (1,2)$.
This concludes the proof of the second part.
Note that alternatively we could have interpreted the integral $a_{\bar{X}}(\cdot) - \tilde{a}_X(\cdot)$ as a generalised hypergeometric function and study its series expansion which is more implicit.
We preferred this more explicit way as it seems more feasible to generalise to higher dimensions.

\noindent
\textbf{Case (ii) $p_X(\cdot)$ locally or asymptotically repairable}\\
Here we follow a similar idea as in case (ii).
Write again 
\begin{equation*}
	a_X(x) 
=
	\left( a_X(x)- \tilde{a}_{\bar{X}}(x) \right) 
	-\left( a_{\bar{X}}(x) -  \tilde{a}_{\bar{X}}(x)\right) 
	+\tilde{a}_{\bar{X}}(x). 
\end{equation*}

The last two terms are exactly the same as in the proof of (i).
 However, the first term behaves differently due the presence of  $\mu_2 |\theta|^2$.
We have that
\begin{equation}
	\label{eq:idea-thm-Greengen}
	\frac{1}{(1-\phi_X(\theta))}-\frac{1}{\mu_\alpha |\theta|^\alpha}	
	=
	\frac{h_X(\theta)}{\mu_\alpha |\theta|^{\alpha}(1-\phi_X(\theta))}
	=
	\mc O \left( |\theta|^{2-2\alpha} \right) 
\end{equation}
as $ |\theta|\rightarrow 0$,
which blows up  slower than $\mc O(|\theta|^{-\alpha})$ for any $\alpha < 2$.
The main idea is to perform a telescopic sum together with expression \eqref{eq:idea-thm-Greengen} until we get a function in $L^1(\bb T)$, which will require exactly $m_\alpha$ iterations.

Note that, in this proof we are only interested in characterising the potential kernel up to a constant order, therefore, we will not need to use information on the degree of continuity of a remainder term as in previous proofs.
Instead, we will compute the first $m_\alpha$ terms by hand an use that the remainder is in $L^1(\bb T)$, for which an application of the Riemann-Lebesgue Lemma \cite[Proposition 3.3.1]{grafakos2008classical} will be enough.

Let
\begin{align*}
	 a_X(x) 
	 &
	 - \tilde{a}_{\bar{X}}(x)
=
	\frac{1}{2\mu_\alpha \pi} 
	\int_{-\pi}^{\pi} \frac{h_X(\theta)}{|\theta|^\alpha (1-\phi_X(\theta))}
	\left( \cos \left(\theta x \right)-1 \right) d\theta.
\end{align*}

For $\alpha < 3/2$ we have that $m_\alpha=0$  and $\tilde{h}_X(\cdot):=\frac{h_X(\cdot)}{|\cdot|^\alpha\left( 1-\phi_X(\cdot) \right)}$ is in $L^1(\bb T)$.
Indeed,
\begin{align*}
	 a_X(x)  
	 - \tilde{a}_{\bar{X}}(x)
	=
	\int_{-\pi}^\pi \tilde{h}_X(\theta)
	\cos \left(\theta x \right)d\theta
	-
	\int_{-\pi}^\pi \tilde{h}_X(\theta)
	d\theta.
\end{align*}
The second term on the r.h.s.~is a constant, whereas the first vanishes as $|x|\rightarrow \infty$ as before.

For the case $\alpha \in (\frac{3}{2}, \frac{5}{3})$ the proof is analogous to the proof of (i): we compare the integral to its counterpart with $1-\phi_X(\theta)$ substituted by $\mu_\alpha|\theta|^\alpha$ in the denominator.
Notice that we have not yet covered the case $\alpha=3/2$ which is given at the end of the proof.
Here we have:
\begin{align*}
	a_X(x) - \tilde{a}_{\bar{X}}(x)
&:=
	\underbrace{\frac{\mu_2}{2(\mu_\alpha)^2 \pi} 
	\int_{-\pi}^{\pi}
	\frac{h_X(\theta)}{|\theta|^{2\alpha}}
	\left( \cos \left(\theta x \right) -1 \right) d\theta}_{I(x)}
\\ & \phantom{=}+
\underbrace{\frac{1}{2\mu_\alpha\pi} 
	\int_{-\pi}^{\pi} 
	\left(
		\frac{h_X(\theta)}{|\theta|^\alpha (1-\phi_X(\theta))}
	-\frac{\mu_2 h_X(\theta)}{\mu_\alpha|\theta|^{2\alpha})}
	\right)
	\left( \cos \left(\theta x \right) -1 \right) d\theta}_{R_0(x)}.
\end{align*}
The last remainder term $R_0(x)$ is of order $\mc O(1)$ as $|x|\longrightarrow
\infty$ for any $\alpha < 2$, again due to the fact that we can interpret it as
the Fourier transform of a $L^1(\bb T)$ function.

Since we assumed $\alpha > \frac{3}{2}$, 
$\theta\mapsto |\theta|^{2-2\alpha}\left( \cos \left(\theta x \right) -1 \right)$ is in $L^1(\bb R)$
and therefore
\begin{align*}
	I(x)
& =	
	|x|^{2\alpha-3} \frac{\mu_2}{2 (\mu_\alpha)^2 \pi} \int_{-\pi x}^{\pi x}
	|\theta|^{2-2\alpha} \left( \cos(\theta)-1 \right)d\theta
\\ & \phantom{=} +
	\frac{\mu_2}{2\mu_\alpha\pi} 
	\int_{-\pi}^{\pi} 
	\frac{h_X(\theta)- |\theta|^2}{|\theta|^{2\alpha}}
	\left( \cos \left(\theta x \right) -1 \right) d\theta
\\ &=
	\underbrace{|x|^{2\alpha-3} \frac{\mu_2}{2 (\mu_\alpha)^2 \pi} \int_{-\infty}^{\infty}
	|\theta|^{2-2\alpha} \left( \cos(\theta)-1 \right)d\theta}_{I_1(x)}
\\ & \phantom{=} -
\underbrace{|x|^{2\alpha-3}\frac{\mu_2}{2 (\mu_\alpha)^2 \pi}
	\int_{\bb R \setminus [-\pi x, \pi x]}
	|\theta|^{2-2\alpha} \left( \cos(\theta)-1 \right)d\theta}_{R_{1,1}(x)}
\\ & \phantom{=} +
\underbrace{\frac{\mu_2}{2\mu_\alpha\pi} 
	\int_{-\pi}^{\pi} 
	\frac{h_X(\theta)-|\theta|^2}{|\theta|^{2\alpha}}
	\left( \cos \left(\theta x \right) -1 \right) d\theta}_{R_{1,2}(x)}.
\end{align*}
Both terms $R_{1,1},R_{1,2} = \mc O(1)$ as $|x| \longrightarrow \infty$, since
\begin{equation*}
	|x|^{2\alpha-3}
	\left| 
		\int_{\bb R \setminus [-\pi x, \pi x]}
		|\theta|^{2-2\alpha} \left( \cos(\theta)-1 \right)d\theta
	\right|
	=
	\mc O (1),
\end{equation*}
for any $\alpha <2$. 
More generally, let $\alpha \in (1,2)$ and $2/ (2-\alpha) \not \in \bb N$, we write 
\begin{align}
	\label{eq:genGreen-exp}
	\int_{-\pi}^{\pi} \frac{h_X(\theta)
	}{|\theta|^\alpha (1-\phi_X(\theta))}
	 & \left( \cos \left(\theta x \right)-1 \right) d\theta
\\&  =
\nonumber
	\sum_{m=1}^{m_\alpha}
\underbrace{\int_{-\pi}^{\pi} 
	\frac{\mu_2^m}{\mu_\alpha^m} \frac{\left( h_X(\theta) \right)^m}{|\theta|^{(m+1)\alpha}}
	\left( \cos \left(\theta x \right)-1 \right) d\theta}_{I_m(x)}
\\&\phantom{=}+
\underbrace{\int_{-\pi}^{\pi} 
	\frac{\mu_2^{m_\alpha+1}}{\mu_\alpha^{m_\alpha+1}} 
	\frac{\left( h_X(\theta) \right)^{m_\alpha+1}}{|\theta|^{( m_\alpha+1 ) \cdot\alpha}(1-\phi_X(\theta))}
	\left( \cos \left(\theta x \right)-1 \right) d\theta}_{R(x)}
\nonumber
\\&  =
\nonumber
	\sum_{m=1}^{m_\alpha}
	I_m(x)
+
	R(x).
\nonumber
\end{align}
We chose $m_\alpha = \lceil \frac{\alpha-1}{2-\alpha} \rceil-1$ as the minimal value of
$m$ such that 
\[
	\frac{(h_X(\theta))^{m_\alpha +1} }{(1-\phi_X(\theta))|\theta|^{m_\alpha+1}} \in L^1(\bb T).
\]
Analogously as before we argue that $R(x)=\mc O(1)$ as $|x| \longrightarrow \infty$.

Finally, for $m \le m_\alpha$ we have 
\[
	\frac{h_X^m(\theta)}{\mu_\alpha^m |\theta|^{m \alpha} (1-\phi_X(\theta))}
	=
	\frac{\mu_2^m}{\mu_\alpha^{m+1}}|\theta|^{m(2-\alpha)-\alpha} + 
	\mc O \left( |\theta|^{m(2-\alpha)-1} \right),
\]
and as $\alpha<2$, we have that ${m(2-\alpha)-1}>-1$, using a 
change of variable we get
\begin{align*}
	I_m(x)
&=
	\frac{\mu_2^m}{\mu_{\alpha}^{m}}
	\int_{-\pi}^\pi 
	|\theta|^{m(2-\alpha)-\alpha} \left( \cos(\theta x)-1 \right)
	d\theta
	+
	\mc O(1)
\\&=
|x|^{(\alpha-1) - m(2-\alpha)}
	\frac{\mu_2^m}{\mu_{\alpha}^{m}}
	\int_{-\infty}^\infty 
	|\theta|^{m(2-\alpha)-\alpha} \left( \cos(\theta )-1 \right)
	d\theta
\\ & \phantom{=}-
	\frac{\mu_2^m}{\mu_{\alpha}^{m}}
	\int_{ \bb R \setminus [-\pi |x|, \pi |x|]}
	|\theta|^{m(2-\alpha)-\alpha} \left( \cos(\theta x)-1 \right)
	d\theta
	+
	\mc O(1).
\end{align*}
Where the first integral in the second line is finite because $m < m_\alpha$.
Again, notice that the last integral is of order $\mc O(1)$ as $|x|\longrightarrow \infty$.

Finally, if $2/ (2-\alpha) \in \bb N$ (which includes the $\alpha=3/2$ case), we have that 
\[
	\frac{(h_X(\theta))^{m_\alpha +1} }{(1-\phi_X(\theta))|\theta|^{m_\alpha+1}} - 
	\frac{\mu_2^{m_\alpha+1}}{\mu_\alpha^{m_\alpha+2}|\theta|} \in L^{1}(\bb T).
\]
Now, we proceed like before, but also taking into account the contribution of the integral
\[
	\frac{1}{2\pi}\int_{-\pi}^\pi \frac{\cos (x\theta)-1}{|\theta|} d\theta
	=
	\frac{1}{\pi}\int_{0}^{\pi |x|} \frac{\cos (\theta)-1}{\theta} d\theta
\]
and using Lemma \ref{lem:app-int}.
This concludes the proof.
\end{proof}

\begin{proof}[Proof of Theorem \ref{thm:green-1}]
We will only prove the repaired case, as the other case is just an adaptation of 
the arguments of Theorem \ref{thm:Green} case \textit{(ii)} together
with the considerations we will present here.

Instead of comparing $a_X(\cdot)$ and $a_{\bar{X}}(\cdot)$ and $\tilde{a}_{\bar{X}}(\cdot)$, we 
we will only compare $a_X(\cdot)$ and $\tilde{a}_{\bar{X}}(\cdot)$. That is, we have
\[
	a_{X}(x) := \frac{1}{2\pi} \int_{-\pi}^\pi \frac{1}{1-\phi_X(\theta)} (\cos(\theta x)-1)d\theta
\]
and we define 
\[
	\tilde{a}_{\bar{X}}(x) := \frac{1}{2\pi} \int_{-\pi}^\pi \frac{1}{\mu_1|\theta|} (\cos(\theta x)-1)d\theta.
\]
Write now
\[
	a_{X}(x) :=\tilde{a}_{\bar{X}}(x)+\left(a_{X}(x)-\tilde{a}_{\bar{X}}(x)\right).
\]
To evaluate the second term, we use a very similar approach to the one in the proof of Theorem 
\ref{thm:Green}. Using the second part of the statement of Lemma \ref{lem:app-holder-quocient}, we get
$g:\theta \mapsto \frac{1}{\mu_1|\theta|} - \frac{1}{1-\phi_X(\theta)}$ is in $C^{0,1/3-}(\bb T)$.
Indeed, by writing 
\begin{equation*}
	g(\theta) 
	=
	\frac{f_1(\theta)}{\mu_1-f_2(\theta)}
\end{equation*}
where $f_1(\theta):= h_X(\theta)/|\theta|^2$ and $f_2(\theta)=h_X(\theta)/|\theta|$ and applying the second statement of Lemma~\ref{lem:app-holder-quocient} for $f_1$ and $f_2$, we get that $f_1 \in C^{0,1/3-}(\bb  T)$ and $f_1 \in C^{0,2/3-}(\bb  T)$, by taking the minimum of the regularities, we recover the desired result.

It remains to evaluate $\tilde{a}_X(x)$. Note that 
\[
	\tilde{a}_{\bar{X}}(x) = \frac{1}{\pi \mu_1} \int_0^{\pi |x|} \frac{\cos(\theta)-1}{\theta} d\theta. 	
\]
Again,	using Lemma \ref{lem:app-int}, we conclude the result.
\end{proof}

\begin{proof}[Proof of Theorem \ref{thm:green-less1}]
This proof is similar to the one of Theorem~\ref{thm:Green}. By writing
\begin{equation*}
	g_X(x)=\frac{1}{2\pi} \int_{\bb T} \frac{1}{1-\phi_X(\theta)} \cos(\theta x)d\theta
\end{equation*}
and comparing it to
\begin{equation*}
	g_{\bar{X}}(x)=\frac{1}{2\pi} \int_{\bb R} \frac{1}{\mu_\alpha |\theta|^\alpha} \cos(\theta x)d\theta
	\quad
	\text{ and }
	\quad
	\tilde{g}_{\bar{X}}(x)=\frac{1}{2\pi} \int_{-\pi}^\pi \frac{1}{\mu_\alpha |\theta|^\alpha} \cos(\theta x)d\theta,
\end{equation*}
we can obtain the error bound by using the second statement of Lemma~\ref{lem:app-holder-quocient}. 
\end{proof}

\section{Fluctuations of Gaussian Fields driven by admissible random walks}
\label{sec:second-order-conv} 
\subsection{Proof of Theorem~\ref{thm:converge-of-fields-elip}}\label{subsec:elliptic-case}


Before we proceed to the proof, we define the required coupling.
To do so, we define $\{\xi^m(x)\}_{m \in \bb N, x \in \bb T_m}$ by taking
\begin{equation}\label{def:discrete-wnoise} 
  \xi^m(x):=\frac{m^{1/2}}{2\pi} \<\xi,  \1_{B_{2\pi/m}(x)}\>,
\end{equation}
where $\xi$ is the same realisation of the white-noise used in the definition of $\Xi_\alpha$.
This allows us to define $\Xi^m - \Xi_\alpha$.

It is easy to show that $\{{\bf e}_k^m\}_{k \in \bb Z_m}$ forms an orthonormal basis of eigenfunctions of $\mc L^m$, where we remember that ${\bf e}_k^m = {\bf e}_k \mid_{\bb  T_m}$ with ${\bf e}_k:= \exp (i k \cdot x)$, $x\in \bb T_m$. 
Moreover, simple computations show that the eigenvalue of $\mc L^m$ associated to ${\bf e}^m_k$ is precisely given by
\begin{equation}\label{eq:eigenvalues}
  -\lambda^{m}_{k} = 1-\phi_X\left(\frac{k}{m}\right)
\end{equation}
for each $k \in \bb Z$.
Consider the Green's function $g_m(\cdot,\cdot)$ associated to the generator in the torus $\bb T_m$ , i.e, the solution of the equation 
\begin{equation*}
  \begin{cases}
 \left( \mc L^m g_{m}(x,\cdot)\right)(y) = \delta_{x}(y) &  \text{ if } y\in \bb T_m, \\
  \sum_{y \in \bb T_m}g_{m}(x,y)=0.
  \end{cases} 
\end{equation*}
Simple computations show that 
\begin{equation*}
  g_{m}(x,y)=\frac{1}{2\pi m}\sum_{k \in \bb Z_m\setminus \{0\}}
  \frac{{\bf e}^{m}_k(x) \overline{{\bf e}^{m}_k(y)}}{\lambda^{m}_k}.
\end{equation*}

We can explicitly write $\Xi^m(x)$ as 
\begin{equation*}
  \Xi^m(x) = g_{m} * (\xi^m - \< \xi^m, 1\>)(x)
\end{equation*}
where $*$ denotes the usual convolution.
Likewise, we can write $\Xi_{\alpha}$ as a convolution (in the distributional sense) as 
\begin{equation*}
  \Xi_{\alpha}(x) = g_\alpha * (\xi - \< \xi, 1\>)(x)
\end{equation*}
where $g_\alpha$ is the Green function associated to $-(-\Delta)_{\bb T}^{\alpha/2}$, which is given by
\begin{equation}\label{eq:cont-green}
  g_{\alpha}(x,y)=\frac{1}{2\pi m}\sum_{k \in \bb Z\setminus \{0\}}
  \frac{{\bf e}_k(x) \overline{{\bf e}_k(y)}}{|k|^{\alpha}}.
\end{equation}

\begin{proof}[Proof of Theorem~\ref{thm:converge-of-fields-elip}]
The proof uses the coupling of the white noise and the explicit Green's function identities provided above.
Define
\begin{equation*}
  \Xi^m_{\alpha,\beta,Err}:=
 m^{\beta-\alpha} \left(\Xi^m-\frac{1}{\mu_\alpha}\Xi_{\alpha}\right)
 -
\frac{\mu_\beta}{(\mu_\alpha)^2}\Xi_{2\alpha-\beta}.
\end{equation*}
We analyse each Fourier coefficient of this field and write
\begin{equation}\label{eq:field-error-split}
\<   \Xi^m_{\alpha,\beta,Err}, {\bf e}_k \>
=
\underbrace{
m^{\beta-\alpha} \<\xi, {\bf e}_k\>
			 \left(
			   \frac{1}{m^\alpha \lambda^{m}_k}
			   -
			   \frac{1}{\mu_\alpha |k|^\alpha}
			   -
			    m^{\alpha-\beta}\frac{\mu_\beta}{(\mu_\alpha)^2 |k|^{2\alpha-\beta}}
			 \right)}_{A_1(m,k)}
+
\underbrace{ m^{\beta-\alpha}\frac{ \<\xi,\tilde{{\bf e}}^m_k-{\bf e}_k\> }{m^\alpha \lambda^{m}_k},}_{A_2(m,k)}
\end{equation}
where $\tilde{{\bf e}}^m_k=\sum_{z \in \bb T_m}{\bf e}_k(z)\1_{B_{2\pi/m}(z)}$. 
We can show that 
\begin{align}\label{eq:A1-fields} 
  \nonumber
 \bb E
 \left(|A_1(m,k)|^2\right)
& = 
  m^{2\beta-2\alpha}
\left(
	\frac{m^\alpha h_X\left(\frac{k}{m}\right)}{
	\mu_\alpha |k|^\alpha\left(\mu_\alpha |k|^\alpha + m^\alpha h_X\left(\frac{k}{m}\right)\right)}
   -
	\frac{\mu_\beta m^{\alpha-\beta}}{(\mu_\alpha)^2 |k|^{2\alpha-\beta}}
\right)^2
\\&  = \nonumber
  |k|^{2\beta-4\alpha}
\left(
\mc O\left(\frac{|k|^{\gamma-\beta}}{m^{\gamma-\beta}}\right)
+
\mc O\left(\frac{|k|^{\beta-\alpha}}{m^{\beta-\alpha}}\right)
\right)^2
\\&  = 
\mc O\left(\frac{|k|^{2\delta_1 + 2\beta - 4\alpha}}{m^{2\delta_2}}\right)
\end{align}
for all $k \in \bb Z_m$ with with $\delta_1:=(\gamma-\beta)\vee(\beta-\alpha)$ and $\delta_2:=(\gamma-\beta)\wedge(\beta-\alpha)$. 

On the other hand, we have that 
\begin{equation}\label{eq:A2-fields} 
  \bb E (|A_2(m,k)|^2)
  \le C
  m^{2(\beta-\alpha-1)}
  |k|^{2-2\alpha}.
\end{equation}

Therefore, we can see that 
\begin{align*}
&
\bb E \left(\|\Xi^m_{\alpha,\beta,Err}\|^2_{H^{-s}}\right)
\\&  = 
  \sum_{k \in \bb Z \setminus \{0\}}
  \bb E(|\< \Xi^m_{\alpha,\beta,Err}, {\bf e}_k \>|^2)
  |k|^{-4s}
\\&  = 
  \sum_{k \in \bb Z_m \setminus \{0\}}
  \bb E(|\< \Xi^m_{\alpha,\beta,Err}, {\bf e}_k \>|^2)
  |k|^{-4s}
  +
  \sum_{k \in \bb Z \setminus \bb Z_m}
  \bb E|\< \Xi^m_{\alpha,\beta,Err}, {\bf e}_k \>|^2
  |k|^{-4s}
\\&  \le 
  Cm^{-2\delta_2}
  +
  Cm^{2(\beta-\alpha) -1}
  +
  C
  m^{2\beta-4\alpha-4s+1}
  +
  C
  m^{2\beta-4\alpha-4s+1}
\end{align*}
which goes to $0$ as long as $s > s_0$ and $\beta < \alpha +1$. 
Applying Chebyshevs's inequality, we recover the convergence in probability.
\end{proof}

\subsection{Proof of Theorem~\ref{thm:converge-of-fields-parab}}
\label{subsec:parabolic-case}

Again, we need to make a few observations before deriving the proof.
By interpreting $\zeta^m$ as a map $t \mapsto \zeta^m(t,\cdot) \in \ell^2(\bb T_m)$, we can look at $\hat{\zeta}^{m}(t,k)$ the $k$-th Fourier coefficient of $\zeta^m(t,\cdot)$.
Notice that due to linearity
\begin{align*}
  d\hat{\zeta}^m(t,k) 
& := 
  \frac{1}{m} 
  \sum_{x \in \bb T_m}
  d\hat{\zeta}^m(t,k) {\bf e}^m_{-k}(x)
\\&  =  
  \frac{1}{m} 
  \left(
  \sum_{x \in \bb T_m}
  m^\alpha\mc L^{m} \zeta^m(t,x)  {\bf e}^m_{-k}(x) \right)dt
  +
  \frac{1}{m} 
  \sum_{x \in \bb T_m}
  {\bf e}^m_{-k}(x)d\xi^m(t,x) 
\\&  =  
  -m^\alpha \lambda^m_k\hat{\zeta}^m(t,k)dt
  +
  d\hat{\xi}^m(t,k) 
\end{align*}
where $\{\hat{\xi}^m(\cdot,k)\}_{k \in \bb Z^m}$ is also collection of i.i.d. Brownian motions.
Notice that $\hat{Z}^m(t,k) = \hat{\zeta}^m (t,k)$. 
Then we use It\^o's formula to get that that the term $\hat{Z}^m(t,k)$ can be written as
\begin{align*}
 \hat{Z}^m(t,k) 
:=
\hat{Z}^m_0(k)e^{-m^\alpha\lambda_k^m t}
+
 \int_{0}^t e^{-m^\alpha\lambda_k^m (t-s)}  
 \hat{\xi}^m(ds,k).
\end{align*}

Now, we construct a coupling between the continuous and discrete versions by taking 
\begin{equation*}
  \xi^m(t,x):= m^{1/2} \< \xi(t,\cdot), \1_{B_m(x)} \>
\end{equation*}
where we are abusing the notation as $\xi^m(t,x)$ is not a function in $t$ (but rather a distribution).
We can write 
\begin{align*}
 \hat{Z}^m(t,k) 
:=
m^{1/2}\hat{Z}^m_0(k)e^{-m^\alpha\lambda_k^m t}
+
 m^{1/2}\xi ( f_{m,k}(t,\cdot,\cdot) )
\end{align*}
where $f_{m,k}(t,s,y):= \1_{[0,t]}(s)e^{-m^\alpha \lambda_k^m (t-s)} \tilde{{\bf e}}^m_k(y)$ and $\tilde{{\bf e}}^m_k$ is the same as in \eqref{eq:field-error-split}. 

By running a similar argument as the elliptic one, we have that 
\begin{equation}\label{eq:conv-parabollic-field}
  m^{\beta-\alpha}\left(\frac{Z^{m}}{m^{1/2}} - Z_\alpha\right)- Z_{Err} \longrightarrow 0 \text{ in probability }
\end{equation}
where the convergence happens in $L^2([0,T],H^{-s})$ for any $s >\max\{2\beta-\alpha,\gamma-\alpha\}$ and for any $T>0$ and $Z_{Err}$ is characterised by its Fourier coefficients as 
\begin{align*}
  \hat{Z}_{Err}(t,k):=
  -\mu_\beta|k|^\beta\hat{Z}_0(k)e^{-\mu_\alpha |k|^\alpha t}t
 - \mu_\beta\int_{0}^t e^{-\mu_\alpha |k|^\alpha (t-s)} |k|^\beta (t-s)\hat{\xi}(ds,k).
\end{align*}

\begin{proof}[Proof of Theorem~\ref{thm:converge-of-fields-parab}]
Due the linearity of the problem, we can deal with the initial condition separately, an analysis that follows similarly to the one of the forcing term.
Therefore, we assume that $Z_0 \equiv 0 $.
Again, we examine each Fourier mode independently and write 
\begin{align*}
  \frac{\hat{Z}^m(t,k)}{m^{1/2}}:= \xi \left( \1_{[0,t]}(\cdot)e^{-m^\alpha \lambda_k^m (t-\cdot)} {\bf e}_k(\cdot)\right) 
  + \xi \left(  \1_{[0,t]}(\cdot)e^{-m^\alpha \lambda_k^m (t-\cdot)} (\tilde{{\bf e}}^m_k(\cdot)-{\bf e}_k(\cdot))\right).
\end{align*}
From this we examine
\begin{align*}
  m^{\beta-\alpha}&\left(\frac{\hat{Z}^{m}(t,k)}{m^{1/2}} - \hat{Z}_\alpha(t,k)\right)- \hat{Z}_{Err}(t,k)
  =
\\ &
   \xi \left( \1_{[0,t]}(\cdot)
	 \left(m^{\beta-\alpha}\left(e^{-m^\alpha \lambda_k^m (t-\cdot)}
	 -
	  e^{-\mu_\alpha|k|^\alpha (t-\cdot)}\right)
	 +
	  \mu_\beta |k|^\beta  e^{-\mu_\alpha|k|^\alpha (t-\cdot)}\right)(t-\cdot)
   {\bf e}_k(\cdot)\right) 
\\ &  \phantom{=}   
  +m^{\beta-\alpha} \xi \left(  \1_{[0,t]}(\cdot)e^{-m^\alpha \lambda_k^m (t-\cdot)} (\tilde{{\bf e}}^m_k(\cdot)-{\bf e}_k(\cdot))\right)
\\&  = 
A_m(t,k) + B_m(t,k). 
\end{align*}
Now we fix $T>0$, we have that the following bound on the second moment of the second term
\begin{align*}
  \int_{0}^T \bb E ( (B_m(t,k))^2 ) dt
& =
  \int_{0}^T \frac{1-e^{-2 m^\alpha \lambda^m_k t}}{2 m^\alpha \lambda^m_k } \|\tilde{{\bf e}}^m_k-{\bf e}_k\|^2_{L^2} dt
\\& \lesssim
  T m^{2(\beta-\alpha-1)} |k|^{2-\alpha}.
\end{align*}
As for the first term, using Taylor expansion we have that 
\begin{align*}
 & \int_{0}^T \bb E (A(t,k)^2 ) dt
 = 
  \int_{0}^T 
   \int_{0}^t 
   e^{-2\mu_\alpha |k|^\alpha(t-s)}
	 \left(
   m^{\beta-\alpha}
   \left(  e^{(-m^\alpha \lambda_k^m + \mu_\alpha |k|^\alpha)(t-s)}
	- 1\right)
	 +
	  \mu_\beta |k|^\beta  (t-s)
	\right)^2
	  ds
  dt
\\&  =
  \int_{0}^T 
   \int_{0}^t 
   e^{-2\mu_\alpha |k|^\alpha(t-s)}
   \left(
	\mc	O \left(m^{\beta-\gamma}|k|^\gamma(t-s)\right)
	+
	\mc	O \left(m^{\alpha- \beta}|k|^{2\beta}
	(t-s)^2\right)
	\right)^2
	  ds
  dt
\\&  \lesssim
(1\vee T^{3})  |k|^{s_1-\alpha} m^{-s_2}
\end{align*}
where $s_1 =\max \{\gamma, 2\beta\} $ and $s_2:=\max\{\gamma-\beta,\beta-\alpha\}$, the proof now follows from another Chebyshev's inequality and a triangular inequality.

\end{proof}
\appendix
\section{Evaluation of some special integrals}
\begin{lemma}\label{lem:ap-R-alpha+}
For $\alpha \in (0,2)\setminus \{1\}$, let $R^\infty_\alpha$  be defined as in
\eqref{def:r-infity+}. Then, there exist real constants $K_1,\dots,K_3$ such that
\begin{equation} \label{asymp-of-R+}
	R^\infty_{\alpha,+}(\theta) = 
	\sum_{k=1}^3 i^k K_k \theta^k
	+ \mc O(|\theta|^{2+\alpha}). 
\end{equation}
\end{lemma}

The proof of this lemma is very similar to the next, which refers to the symmetric case.
However, in the symmetric case, we need to be more careful, as we will be interested in showing that the distribution $p_\alpha(\cdot)$ is not only admissible, but repairable.
That means we will need to control the signs of certain constants.
To avoid repeating ourselves, we will present only the proof of Lemma~\ref{lem:ap-R-alpha}.

\begin{lemma}\label{lem:ap-R-alpha}
For $\alpha \in (0,2)\setminus \{1\}$, we have that $R^\infty_\alpha$ defined in \eqref{def:r-infity}
satisfies 
\begin{equation} \label{asymp-of-R}
	R^\infty_\alpha(\theta) = K_2|\theta|^2
			+\mc O(|\theta|^{2+\alpha})
\end{equation}
where
\begin{align}
\label{eq:asymp-of-R-2}
	K_2
	=
	\frac{1-\alpha}{2} 
	\Bigg(
	&\left( \frac{2^{2-\alpha}-1}{2-\alpha}-\frac{3(2^{1-\alpha}-1)}{2(1-\alpha)}\right)
	\\ &+
	\frac{1}{2\Gamma(\alpha)}
	\sum_{m=1}^\infty (-1)^m (\zeta(m+\alpha)-1) \frac{m \Gamma(m+\alpha)}{\Gamma(m+2)(m+2)}
	\Bigg).
	\nonumber
\end{align}
\end{lemma}

\begin{proof}[Proof of Lemma~\ref{lem:ap-R-alpha} ]
Recall that $\theta >0$,
\begin{align*}
	R^\infty_\alpha = \theta^{1+\alpha}
	\int_{\theta}^\infty
	\Big(\frac{z \sin(z) - (1+\alpha)(1-\cos(z))}{z^{2+\alpha}}\Big)
	P_1\Big(\frac{z}{\theta}\Big)	dz
\end{align*}
and $P_1(x)=\left(x - \lfloor x \rfloor \right)- \frac{1}{2}$.  Note that this
integral is finite. Indeed, one can prove this by observing that $|P(z)|\le
\frac{1}{2}$. We shall now divide the integral in $R_\alpha^\infty$ in two
parts, one going from $\theta$ to $1$ and the other $1$ to $\infty$, as we will
use different techniques to bound them.
\begin{align*}
	R_\alpha^\infty &=
\underbrace{\theta^{1+\alpha}
	\int_{\theta}^1
	\frac{z \sin(z) - (1+\alpha)(1-\cos(z))}{z^{2+\alpha}}
	P_1\Big(\frac{z}{\theta}\Big) dz}_{I_1}
	\\ &\phantom{=}+
\underbrace{\theta^{1+\alpha}
	\int_{1}^\infty
	\frac{z \sin(z) - (1+\alpha)(1-\cos(z))}{z^{2+\alpha}}
	P_1\Big(\frac{z}{\theta}\Big) dz}_{I_2}.
\end{align*}
We start by analysing $I_2$ and proving that $I_2 = \mc O(|\theta|^{2+\alpha})$, 
\begin{align*}
	I_2 
&=
	\theta^{1+\alpha}\int_{1}^\infty
	\frac{z \sin(z)- (1+\alpha)(1-\cos(z))}{z^{2+\alpha}}
	P_1\Big(\frac{z}{\theta}\Big) dz.
\end{align*}
For convenience, we assume that
$\theta^{-1}\in \bb N$.  To treat the general case we need 
to compare the expressions between for $\theta^{-1}$ and  
$\lfloor\theta^{-1}\rfloor$. 

In this case, we can write the
integral above as 
\begin{align*}
	I_2 
&=
\theta^{1+\alpha}	\sum_{k=1/\theta}^\infty \int_{k\theta}^{(k+1)\theta}
	g(z)\left(\frac{z}{\theta}-k-\frac{1}{2} \right)dz,
\end{align*}
where $g(z):=\frac{z \sin(z)- (1+\alpha)(1-\cos(z))}{z^{2+\alpha}}$. Now, we will
use that $\int_{k\theta}^{(k+1)\theta} P_1\left(\frac{z}{\theta}\right)dz =0$ and 
sum and subtract the term $g(k\theta)$ in each term of the summands. Hence
\begin{align*}
	|I_2| 
&=
	|\theta|^{1+\alpha}
	\left |	
	\sum_{k=1/\theta}^\infty \int_{k\theta}^{(k+1)\theta}
	(g(z)-g(k\theta))\left(\frac{z}{\theta}-k-\frac{1}{2} \right)dz,
	\right|
\\ & \le 
	|\theta|^{1+\alpha}
	\sum_{k=1/\theta}^\infty  \sup_{y \in [k\theta,(k+1)\theta]} |g^\prime(y)|
	\int_{k\theta}^{(k+1)\theta}
	|z-k\theta|\left|\frac{z}{\theta}-k-\frac{1}{2} \right|dz,
\\ & \le 
\frac{1}{4}	|\theta|^{3+\alpha}\sum_{k=1/\theta}^\infty  \sup_{y \in [k\theta,(k+1)\theta]} |g^\prime(y)|,
\end{align*}
where we used in the second inequality both a change of variables and
that $|z-k\theta|\le \theta$.

For $z>0$, we have
\[
	g^{\prime}(z)=\frac{\cos(z)}{z^{1+\alpha}}-2(1+\alpha)\frac{\sin(z)}{z^{2+\alpha}}+(1+\alpha)(2+\alpha)\frac{1-\cos(z)}{z^{3+\alpha}}
\]
and therefore there is a constant $C_\alpha$ that only depends on $\alpha$, such 
\[
	|g^{\prime}(z)| \le \frac{C_\alpha}{z^{1+\alpha}}
\]
which implies
\[
	\sup_{[k\theta,(k+1)\theta]}|g^\prime(z)| \le \frac{C_\alpha}{(\theta k )^{1+\alpha}}.
\]
We can now use this in the estimate of $|I_2|$ to get 
\[
	|I_2|
\le	
	C \theta^{2}
	\sum_{k=1/\theta}^\infty \frac{1}{k^{1+\alpha}} 
\lesssim
	 |\theta|^{2+\alpha}
\]
and $I_2 = \mathcal{O}(|\theta|^{2+\alpha})$.

Now, for $I_1$, we use Taylor expansion of the function $h(z) = z \sin (z) - (1+\alpha)(1-\cos (z))$ to get
\[
	h(z)=\frac{1-\alpha}{2}z^2 -\frac{3-\alpha}{24}z^4 + r(z),
\]
where $r(z)=\mc O(z^6)$. We get
\begin{align*}
	I_1
	&=
	\theta^{1+\alpha}
	\frac{1-\alpha}{2}
	\int_{\theta}^1
	\frac{1}{z^{\alpha}}
	P_1\Big(\frac{z}{\theta}\Big) dz
	-
	\theta^{1+\alpha}
	\frac{3-\alpha}{24}
	\int_{\theta}^1
	z^{2-\alpha}
	P_1\Big(\frac{z}{\theta}\Big) dz
	\\ & \phantom{=}
	+
	\theta^{1+\alpha}
	\int_{\theta}^1 
	\frac{r(z)}{z^{2+\alpha}}	
	P_1\Big(\frac{z}{\theta}\Big) dz
	\\ &= \frac{1-\alpha}{2}I_{1,1}-\frac{3-\alpha}{24}I_{1,2}+I_{1,3}.
\end{align*}
Again we examine each of the terms separately. We start with the last one. 
For this, notice that $r(\cdot)$ is a $C^\infty([-1,1])$ function, as it is the difference 
of two such functions. Moreover, we know that 
$\tilde{r}(z):=\left| \frac{r(z)}{z^{2+\alpha}} \right|$ 
and therefore, applying Lemma \ref{lem:app-holder-quocient} we have that $\tilde{r}(\cdot)$ is 
in $C^{0, \frac{4-\alpha}{6}-}([-1,1])$. Now we can proceed like we did for $I_2$ to get that 
$I_{1,3}$ is of order $\mc O(\theta^{2+\alpha})$.

The first integral $I_{1,1}$ can be written as, again assuming that 
$\theta^{-1} \in \bb N$, 
\begin{align*}
	I_{1,1}
	&
	=
	\theta^{1+\alpha}
	\sum_{k=1}^{\lfloor \frac{1}{\theta} \rfloor -1 }
	\int_{k\theta}^{(k+1)\theta}
	\frac{1}{z^{\alpha}}
 	\Big(\frac{z}{\theta}- k - \frac{1}{2}\Big)dz
	\\ &=
	\theta^{2}
	\sum_{k=1}^{\lfloor \frac{1}{\theta} \rfloor -1 }
	k^{2-\alpha}\left[
	\frac{\Big(1+\frac{1}{k}\Big)^{2-\alpha}-1}{2-\alpha}
	-
	\Big(1+\frac{1}{2k}\Big)
	\frac{\Big(1+\frac{1}{k}\Big)^{1-\alpha}-1}{1-\alpha}\right].
\end{align*}
We now split the terms with $k=1$ and $k>1$,
\begin{align}\label{eq:I11} 
	I_{1,1}
	&
	\nonumber
	=
	\theta^2 \left( \frac{2^{2-\alpha}-1}{2-\alpha}-\frac{3(2^{1-\alpha}-1)}{2(1-\alpha)}\right)
	\\&+
	\theta^{2}
	\sum_{k=2}^{\lfloor \frac{1}{\theta} \rfloor -1 }
	k^{2-\alpha}\left[
	\frac{\Big(1+\frac{1}{k}\Big)^{2-\alpha}-1}{2-\alpha}
	-
	\Big(1+\frac{1}{2k}\Big)	
	\frac{\Big(1+\frac{1}{k}\Big)^{1-\alpha}-1}{1-\alpha}\right].
\end{align}
Use now the full Taylor series of both $(1+x)^{2-\alpha}$ and $(1+x)^{1-\alpha}$ where we are taking $x= \frac{1}{k} \in(0,1)$ to explore the cancellations.
For a fixed $k>1$, the expression inside the square brackets in the last summation is
\begin{align*}
	\frac{1}{2-\alpha}\sum_{j=1}^{\infty} \frac{(2-\alpha)_j}{j!} \frac{1}{k^j}
	-
	\frac{1}{1-\alpha}\sum_{j=1}^{\infty} \frac{(1-\alpha)_j}{j!} \frac{1}{k^j}
	-
	\frac{1}{2(1-\alpha)}\sum_{j=1}^{\infty} \frac{(2-\alpha)_j}{j!} \frac{1}{k^{j+1}}
\end{align*}
where $(x)_j:= x (x-1)\dots (x-j+1)$.
By grouping the powers of $\frac{1}{k}$ together, we can check by hand that the coefficients of $\frac{1}{k}$ and $\frac{1}{k^2}$ are zero.
Moreover, by a simple change of variable on the last sum, we have that the sum above equals
\begin{align*}
	\sum_{j=3}^{\infty} 
	\left(\frac{(1-\alpha)_{j-1}}{j!}
	-
	\frac{(-\alpha)_{j-1}}{j!}
	-\frac{(-\alpha)_{j-2}}{2(j-1)!}
	\right)
	\frac{1}{k^j},
\end{align*}
where we used that $\frac{(x)_j}{x}=(x-1)_{j-1}$ and $(x)_{j+1}=(x)_j (x-j)$. 
Rewriting this expression in terms of Gamma functions, we have that the last term of \eqref{eq:I11} is equal to
\begin{equation}\label{eq:sumI2}
	\theta^{2}
	\sum_{k=2}^{\lfloor \frac{1}{\theta} \rfloor -1 }
	\sum_{j=3}^\infty
	k^{2-\alpha-j}
	\frac{(j-2)\Gamma(1-\alpha)}{2 j! \Gamma( - \alpha - j +3)}.
\end{equation}
From the reflection formula for the Gamma function and a change of variables
$m = j-2$, we get
\begin{align*}
\eqref{eq:sumI2}
	=
	\frac{\theta^2}{2\Gamma(\alpha)}
	\sum_{k=2}^{\lfloor \frac{1}{\theta} \rfloor -1 }
	\sum_{m=1}^\infty (-1)^m k^{-\alpha-m} \frac{m \Gamma(m+\alpha)}{(m+2)!}.
\end{align*}

Now, using Euler-Maclaurin again, one can easily prove that for $\alpha \in (0,2)$ and $m\ge 1$, 
\begin{align}
	\label{eq:bound-inc-zeta}
\left	|\sum_{k=2}^{\lfloor \frac{1}{\theta} \rfloor -1}k^{-\alpha-m}
	-
	(\zeta(m+\alpha)-1) + \frac{\theta^{m+\alpha-1}}{m+\alpha-1} \right|
	\le C \theta^{m+\alpha}
\end{align}
where the constant $C$ does not depend on $m$ or $\alpha$. Therefore there exist an explicit constant $K_2$ 
\[
	I_{1,1} = K_2 |\theta|^2 + \mathcal{O}(|\theta|^{2+\alpha}).
\]

Finally, we can show in an analogous way that  $I_{1,2} = \mathcal{O}(|\theta|^4)$. 
For the case $\alpha=1$ we proceed in a similar way.
We need to evaluate
\[
	R^\infty_1 = \theta^{2}
	\int_{\theta}^\infty
	\Big(\frac{z \sin(z)- 2(1-\cos(z))}{z^{3}}\Big)
	P_1\Big(\frac{z}{\theta}\Big)	dz.
\]
Using similar ideas as before and the fact that
$z \sin(z)- 2(1-\cos(z))=\mc O(z^4)$ when
$|z| \to 0$ instead of the order $\mc O (z^2)$ that we got for the case
$\alpha \in (1,2)$ we conclude the proof.
\end{proof}

It is worth explaining why we do not simply use the triangular inequality to bound the series representation by taking 
\begin{align*}
	|I_{1,1}|
	&
	\le
	\theta^{1+\alpha}
	\sum_{k=1}^{\lfloor \frac{1}{\theta} \rfloor -1 }
	\int_{k\theta}^{(k+1)\theta}
	\left|
	\frac{1}{z^{\alpha}}
 	\Big(\frac{z}{\theta}- k - \frac{1}{2}\Big)
	\right|
	dz,
\end{align*}
or something similar.
This strategy is simply not good enough to guarantee that $\mu_2$ studied in the next lemma is positive for all $\alpha$.
Indeed, applying the bounds above would only be enough to show that $\mu_2$ is positive in an interval of the form $(\alpha_0,2)$ with $\alpha_{0} > 0$.
Although more technical, we preferred to keep a consistent approach for the proof to avoiding having yet more cases.

\begin{lemma}\label{lem:Right-sign}
	For $\alpha \in (0,2)\setminus\{1\}$, the constant $\mu_2$ defined in \eqref{eq:def-mu2} is positive.
\end{lemma}
\begin{proof}
We start by focusing on the case $\alpha>1$. Recall the expression \eqref{asymp-of-R} for $K_2$. 
As $\alpha >1$, for $m \ge 1$, we have $m+\alpha >2$ and therefore
\begin{align*}
	\zeta(m+\alpha)-1
	&=
	\frac{1}{2^{m+\alpha}} + \sum_{k\ge 3} \frac{3^{m+\alpha}}{3^{m+\alpha}} \frac{1}{k^{m+\alpha}}
	\\ & \le
	 \frac{1}{2^{m+\alpha}} + \frac{1}{3^{m+\alpha}} \sum_{k\ge 3}  \Big(\frac{3}{k}\Big)^2
	\\ & \le
  \frac{1}{2^{m+\alpha}}\Bigg(1+ 9\Big(\zeta(2)-\frac{5}{4}\Big)\Bigg)
	\le \frac{5}{2^{m+\alpha}},
\end{align*}
where $\zeta(z)$ is the zeta-function.
Moreover, using Gautschi's inequality for the ratio of two Gamma functions, see e.g.~\cite{Gau}, we can write
\[
	(m+2)^{\alpha-2}<  \frac{\Gamma(m+\alpha)}{\Gamma(m+2)} <(m+1)^{\alpha-2} < m^{\alpha-2}.
\]
The upper bound on $K_2$ will follow from the lower bound on $\frac{K_2}{1-\alpha}$. We remove all even summands $m$ in the definition of $K_2$ and bound further
\begin{align*}
	\frac{2 K_2}{1-\alpha}	
& \ge
	\Bigg(
	\Bigg( \frac{2^{2-\alpha}-1}{2-\alpha}-\frac{3(2^{1-\alpha}-1)}{2(1-\alpha)}\Bigg)
	\\ &
	\phantom{\left( = \right)}-
	\frac{1}{2\Gamma(\alpha)}
	\sum_{m=0}^\infty (\zeta(2m+1+\alpha)-1) 
	\frac{(2m+1) \Gamma(2m+1+\alpha)}{\Gamma(2m+3)(2m+3)}
	\Bigg)
\\& \ge 
	\Bigg(
	\Bigg( \frac{2^{2-\alpha}-1}{2-\alpha}-\frac{3(2^{1-\alpha}-1)}{2(1-\alpha)}\Bigg)
	-
	\frac{5}{2\Gamma(\alpha)}
	\sum_{m=0}^\infty  \frac{(2m+2)^{\alpha-2}}{2^{2m+1+\alpha}}
	\Bigg)
\\& \ge 
	\Bigg(
	\Big( \frac{2^{2-\alpha}-1}{2-\alpha}-\frac{3(2^{1-\alpha}-1)}{2(1-\alpha)}\Big)
	-
	\frac{5}{12\Gamma(\alpha)}
\Bigg).
\end{align*}

Call  $u: (0,2) \rightarrow \mathbb{R}$ the map 
\begin{equation} \label{eq:bound-u}
t \mapsto \frac{1-t}{2} 	\Bigg(
	\Bigg( \frac{2^{2- t}-1}{2- t}-\frac{3(2^{1-t}-1)}{2(1-t)}\Bigg)
	-
	\frac{5}{12\Gamma(t)}
	\Bigg)
\end{equation}
which is increasing for $t>1$ and simple analysis shows that $u(t)$ is bounded from above by $\frac{1}{4}$.
Now we collect all previous contributions to the constant $\mu_2$ and show
that the sum above cannot flip the sign. This concludes that 
\begin{align*}
	\mu_2
	&=
2c_\alpha	\Bigg(
		\frac{1}{2(2-\alpha)} - \frac{1}{4} - K_2
	\Bigg)  > \frac{(\alpha-1) c_{\alpha}}{2-\alpha}
\end{align*}
 is positive for $\alpha>1$.

For $\alpha<1$, the strategy is similar, only this time, we proceed to get a function 
$\tilde{u}(\cdot)$ similar to \eqref{eq:bound-u}
but bounding $\frac{K_2}{2(1-\alpha)}$ from below (as $1-\alpha$ is now positive). 
\end{proof}

\begin{lemma} \label{lem:app-int}
Let $z \in [1,\infty)$ and define 
\[
	\Cin(z):= \int_0^z \frac{1-\cos(t)}{t}dt.
\]
We have that
\[
	\Cin(z) =  \log (z) + \gamma + \mc O (z^{-1}).
\]
as $z \longrightarrow \infty $ where $\gamma$ is the Euler-Mascheroni constant
\end{lemma}
\begin{proof}
By defining 
\[
	\Ci(z) := -\int_z^\infty \frac{\cos(t)}{t} dt
\]
the linearity of the integral implies that
\[
	\Cin(z)= \log (z)  - \Ci(z) + \int_1^\infty \frac{\cos t}{t} dt
	+ \int_0^1 \frac{1-\cos t}{t}dt.
\]
The exact value of the sum of the two integrals is not relevant for us,
but it is known to be $\gamma$. Therefore,
\[
	\Cin(z) = -\Ci(z) + \log (z) + \gamma.
\]
Finally we conclude the proof by noting that trivially $\Ci = \mc O(z^{-1})$ as $z \to \infty.$
\end{proof}
\section{Continuity estimates}

\begin{lemma}\label{lem:app-holder-quocient}
	Let $f \in C^{1,\beta}(I)$ for a closed interval $I$ containing the origin.
	Additionally, suppose that 
	\[
		f(x) = \mc O\left(  |x|^{\beta_0} \right)
		\text{ as } |x| \longrightarrow 0
	\]
	for some $\beta_0 \ge 1 + \beta$.
	Let $1<\beta_1 < \beta_0 $ and define the function 
	\[
		h(x):= \frac{f(x)}{|x|^{\beta_1}}.
	\]
	Then we have that the function $h$ is in $C^{0,\bar{\beta}}(I)$ where $\bar{\beta}= \frac{\beta_0-\beta_1}{\beta_0-\beta}$.
	If instead, we have that $f \in C^{0,\beta}(I)$ for some $\beta \in (0,1)$, and $0<\beta_1<\min\{1, \beta_0\}$, we get that $h \in C^{0,\bar{\beta}}(I)$ with $\bar{\beta} := \min\{\beta(1-\beta_1/\beta_0),\beta_0-\beta_1,1,\beta_1\}$.
\end{lemma}

\begin{proof}
We will prove the first claim, the second can be proved analogously.
Let $x,y \in I$ and assume, without loss of generality, that $|x|<|y|$,
\begin{align*}
	\left| 
		\frac{f(x)}{|x|^{\beta_1}}
		-
		\frac{f(y)}{|y|^{\beta_1}}
		\pm 
		\frac{f(x)}{|y|^{\beta_1}}
	\right|
&=
	\left| 
		\frac{f(x)}{|x|^{\beta_1}}
		\left (\frac{|y|^{\beta_1} - |x|^{\beta_1}}{|y|^{\beta_1}} \right )
		+
		\frac{f(x)-f(y)}{|y|^{\beta_1}}
	\right|
\\& \lesssim
		|x|^{\beta_0-\beta_1}
		\frac{\left||y|^{\beta_1} - |x|^{\beta_1}\right|}{|y|^{\beta_1}}
	+
		\frac{|f(x)-f(y)|}{|y|^{\beta_1}}.
\end{align*}
Now use that for $A,B> C >0$  real numbers and $\delta \in [0,1]$, we have $C \le A^\delta B^{1-\delta}$.
Regarding the first term on the right hand side, notice that
\[
\left||y|^{\beta_1} - |x|^{\beta_1}\right| \lesssim \min\{ |y|^{\beta_1},    |y|^{\beta_1-1} |x-y|\}
\]
so choosing $A=|y|^{\beta_1}, B= |y|^{\beta_1-1}|x-y|$ and $\delta=\beta_0-\beta_1$ we can easily see that
\[
|x|^{\beta_0-\beta_1}
		\frac{\left||y|^{\beta_1} - |x|^{\beta_1}\right|}{|y|^{\beta_1}} \lesssim |x-y|^{\delta} \leq |x-y|^{\bar{\beta}}. 
\]
To bound  the second term, remark that  $|f'(z)| \leq C |y|^{\beta}$ for all
$|z|\leq |y|$ since $f'\in C^{0,\beta}(I)$ and $f'(0)=0$, so
\[
  |f(x)-f(y)| \lesssim \min\{ |y|^{\beta_0}, |y|^{\beta} |x-y|\}
\]
and again choosing $A=|y|^{\beta_0}, B=|y|^{\beta} |x-y|$ and $\delta=\bar{\beta}$ the  claim follows.
\end{proof}

\begin{lemma}
	If $p_X(\cdot)$ is admissible with index $\alpha \in (1,2)$, then $\phi_X(\cdot)$ 
	is in $C^{1,\alpha-1-}(\bb T)$.
	If $p_X(\cdot)$ is admissible with index $1$, then $\phi_X$ is $C^{0,1-}(\bb T)$.
	\label{lemma-app-phi-smooth}
\end{lemma}
\begin{proof}
Notice that $p_X(\cdot)$ being admissible implies that it is in the basin of attraction of an $\alpha$-stable distribution.
Therefore given $\beta \ge 0$ we have $\bb E[|X|^\beta]<\infty$ for $\beta\in (0,\alpha)$ and $p_X(x) \lesssim |x|^{-\alpha+}$.
Now, we just write that $p_X(\cdot)$ as the inverse Fourier transform of
	\[
		\mc F_{\bb T}(\phi_X)(-x) = p_X(x).
	\]
	Then use the classic relations between continuity and 
	decay of Fourier coefficients, see \cite[Proposition 3.3.12]{grafakos2008classical} to conclude the proof.
\end{proof}

\bibliographystyle{abbrv}
\bibliography{library}

\end{document}